\documentclass[a4paper,bej,preprint,reqno]{imsart}
\usepackage[utf8]{inputenc}
\usepackage{amsmath}
\usepackage{amsfonts,amssymb}
\usepackage{dsfont}
\usepackage{graphicx,xcolor}
\usepackage{caption,mathrsfs,booktabs}
\usepackage{nicefrac,a4wide}

\RequirePackage[colorlinks,citecolor=blue,urlcolor=blue,breaklinks=true]{hyperref}
\usepackage[round,sort]{natbib}
\usepackage{cleveref}
\usepackage{autonum}

\newcommand\bB{\boldsymbol B}

\newcommand\bfC{\mathbf C}
\newcommand\bfE{\mathbf E}
\newcommand\bfP{\mathbf P}
\newcommand\bfQ{\mathbf Q}
\newcommand\bfN{\mathbf N}

\newcommand\bfH{\mathbf H}
\newcommand\bfR{\mathbf R}

\newcommand\bfM{\mathbf M}
\newcommand\bL{\boldsymbol L}
\newcommand\bA{\boldsymbol A}
\newcommand\bC{\boldsymbol C}

\newcommand\bW{\boldsymbol W}
\newcommand\bV{\boldsymbol V}
\newcommand\bx{\boldsymbol x}
\newcommand\bv{\boldsymbol v}
\newcommand\bxi{\boldsymbol \xi}
\newcommand\RR{\mathbb R}

\newcommand\LL{\mathbb L}

\newcommand\bfI{\mathbf I}

\newcommand\btheta{\boldsymbol\theta}
\newcommand\bvartheta{\boldsymbol\vartheta}
\def\Tilde{\widetilde}
\def\tilde{\widetilde}
\def\hat{\widehat}
\def\ds{\displaystyle}
\usepackage{amsthm}
\newtheorem{theorem}{Theorem}
\newtheorem{proposition}{Proposition}
\newtheorem{lemma}{Lemma}

\newtheorem{remark}{Remark}

\begin{document}

\begin{frontmatter}

\title{On sampling from a log-concave density using kinetic Langevin diffusions}
\runtitle{Sampling log-concave density by kinetic Langevin}

\begin{aug}
\author{\fnms{Arnak S.}\snm{Dalalyan}\ead[label=e1]{arnak.dalalyan@ensae.fr}}
\and
\author{\fnms{Lionel} \snm{Riou-Durand}\ead[label=e2]{lionel.riou-durand@ensae.fr}}

%\thankstext{T1}{Research partially supported by ANR Parcimonie}
%\thankstext{t2}{First supporter of the project}
%\thankstext{t3}{Second supporter of the project}
\runauthor{Dalalyan, A.S. and Riou-Durand, L.}

\affiliation{ENSAE-CREST}

{\renewcommand{\addtocontents}[2]{}
\address{{5, av. Le Chatelier,}\\
{91129 Palaiseau,}
{France} \\
\printead{e1,e2}}
}
\end{aug}

%\today

\begin{abstract}
Langevin diffusion processes and their discretizations are often used for sampling from a target density. The most convenient framework for assessing the quality of such a sampling scheme corresponds to smooth and strongly log-concave densities defined on $\RR^p$. The present work focuses on this framework and studies the behavior of the Monte Carlo algorithm based on discretizations of the kinetic Langevin diffusion. We first prove the geometric mixing property of the kinetic Langevin diffusion with a  mixing rate that is optimal in terms of its dependence on the condition number. We then use this result for obtaining improved guarantees of sampling using the kinetic Langevin Monte Carlo method, when the quality of sampling is measured by the Wasserstein distance. We also consider the situation where the Hessian of the log-density of the target distribution is Lipschitz-continuous. In this case, we introduce a new discretization of the kinetic Langevin diffusion and prove that this leads to a substantial improvement of the upper bound on the sampling error measured in Wasserstein distance.
\end{abstract}



%\tableofcontents
\end{frontmatter}

\maketitle
%\dominitoc

\itemsep=0pt
\parskip=3pt
\setcounter{tocdepth}{1}
\tableofcontents

\parskip=10pt

\section{Introduction}

Markov processes and, more particularly, diffusion processes are often used in order to solve the problem of sampling from a given density $\pi$. This problem can be formulated as follows. Assume that we are able to generate an arbitrary number of independent standard Gaussian random variables $\xi_1,\ldots,\xi_K$. For a given precision level $\varepsilon>0$ and a given metric $d$ on the space of probability measures, the goal is to devise a function $F_\varepsilon$ such that the distribution $\nu_K$ of the random variable $\vartheta_K = F_\varepsilon(\xi_1,\ldots,\xi_K)$ satisfies $d(\mu_K,\pi)\le \varepsilon$. For solving this task, it is often assumed that we can have access to the evaluations of the probability density function of $\pi$ as well as its derivatives. Among different functions $F_\varepsilon$ having the aforementioned
property, the most interesting are those that require the smallest number of computations.

Markov Chain Monte Carlo methods hinge on random variables $\vartheta_K$ and associated functions $F_\varepsilon$ defined by recursion $\vartheta_{k} = G_\varepsilon(\vartheta_{k-1},\xi_k)$, $k=1,\ldots,K$, where $G_\varepsilon$ is some function of two arguments. For a given target distribution $\pi$, if one succeeds to design a function $G_\varepsilon$ such that the Markov process $\{\vartheta_k;k\in\mathbb N\}$ is ergodic with invariant density $\pi$ then, for large $K$, the distribution of $\vartheta_K$ will be close to $\pi$. Therefore, if the evaluation of $G_\varepsilon$ involves only simple operations, we get a solution of the task of approximate sampling from $\pi$. Of course, it is important to address the problem of the choice of the number of iterations $K$ ensuring that the sampling error is smaller than $\varepsilon$. However, it is even more important to be able to design functions $G_\varepsilon$, often referred to as the update rule,
with desired properties presented above.

Discretization of continuous-time Markov processes is a successful generic method for defining update rules. The idea is to start by specifying a continuous-time Markov process,
$\{L_t:t\ge 0\}$, which is provably positive recurrent and has the target $\pi$ as  invariant distribution\footnote{More generally, one can consider a Markov process having an invariant distribution that is close to $\pi$. }. The second step is to set-up a suitable time-discretization of the continuous-time process.  More precisely, since $\{L_t\}$ is a Markov process, for any step-size $h>0$, there is a mapping $G$ such that $L_{kh} \stackrel{\mathscr D}{=}
G(L_{(k-1)h},\xi_k)$, $k=1,\ldots,K$,  where
$\xi_k$ is a standard Gaussian random variable independent of $L_{(k-1)h}$. This mapping $G$ might not be available in a closed form. Therefore, the last step is to approximate $G$ by a tractable mapping $G_\varepsilon$.
Langevin diffusions are a class of continuous-time Markov processes for which the invariant density is available in closed-form. For this reason, they are suitable candidates for applying the generic approach of the previous paragraph.

Let $m$ and $M$ be two positive constants such that $m\le M$. Throughout this work, we will assume that the target distribution $\pi$ has a density with respect to the Lebesgue measure on $\RR^p$, which is of the form $\pi(\btheta) = C e^{-f(\btheta)}$ for a
function $f$ that
is $m$-strongly convex and with an $M$-Lipschitz gradient. The (highly overdamped) Langevin diffusion having $\pi$ as invariant distribution is defined as a strong solution to the stochastic differential equation
\begin{align}
    d\bL_t = -\nabla f(\bL_t)\,dt + \sqrt{2}\,d\bW_t,\qquad t\ge 0,
    \label{langevin}
\end{align}
where $\bW$ is a $p$-dimensional standard Brownian motion. The update rule associated
to this process, obtained by using the Euler discretization, is given by the
equation
$G_\varepsilon(\bL_{(k-1)h},\bxi_k) = -h \nabla f(\bL_{(k-1)h}) + \sqrt{2h}\, \bxi_{k}$
with $\bxi_k \stackrel{\mathscr D}{=} h^{-1/2}(\bW_{kh}-\bW_{(k-1)h})$ being
a $p$-dimension standard Gaussian vector. The resulting approximate sampling
method is often called Langevin Monte Carlo (LMC) or Unadjusted Langevin Algorithm (ULA).
Its update rule follows from \eqref{langevin}
by replacing the function $t\mapsto \nabla f(\bL_t)$ by its piecewise constant
approximation. Therefore, the behavior of the LMC is governed by the following two
characteristics of the continuous-time process: the mixing rate and the smoothness
of the sample paths. A quantitative bound on the mixing rate allows us to choose a
time horizon $T$ such that the distribution of the random vector $\bL_T$ is within
a distance $\varepsilon/2$ of the target distribution, whereas the smoothness of sample
paths helps us to design a step-size $h$ so that the distribution of the discretized
process at $K=T/h$ is within a distance $\varepsilon/2$ of the distribution of $\bL_T$.
For the LMC, we know that the Langevin diffusion mixes exponentially fast with the
precise rate $e^{-mt}$. In addition, almost all sample paths of $\bL$ are H\"older
continuous of degree $\alpha$, for every $\alpha<1/2$. Combining these properties,
it has been shown that it suffices $K_\varepsilon = O((p/\varepsilon^2)\log(p/\varepsilon^2))$
iterations for the LMC algorithm to achieve an error smaller than $\varepsilon$
(both in total-variation and Wasserstein distances); see \citep{Dal14} for the first nonasymptotic result of this type and \citep{Durmus2,durmus2017,DalKar17} for improved versions of it.

Under the same assumptions on the log-target $f$, one can consider the kinetic
Langevin diffusion, also known as the second-order Langevin process, defined by
\begin{align}\label{kinetic}
    d
		\begin{bmatrix}
		\bV_t\\
		\bL_t
		\end{bmatrix}
		&=
		\begin{bmatrix}
		-(\gamma \bV_t + u\nabla f(\bL_t))\\
		\bV_t
		\end{bmatrix}
		\,dt+\sqrt{2\gamma u}
		\begin{bmatrix}
		\bfI_p\\
		\mathbf 0_{p\times p}
		\end{bmatrix}
		\,d\bW_t,\qquad t\ge 0,
\end{align}
where $\gamma>0$ is the friction coefficient and $u>0$ is the inverse mass.
As proved in \citep[Theorem 10.1]{Nelson}, the highly overdamped Langevin
diffusion  \eqref{langevin} is obtained as a limit of the rescaled
kinetic diffusion $\bar \bL_t = \bL_{\gamma t}$, where
$\bL$ is defined as in \eqref{kinetic} with $u=1$,
when the friction coefficient $\gamma$  tends to infinity.

The continuous-time Markov process $(\bL_t,\bV_t)$
is positive recurrent and its invariant distribution is absolutely continuous with respect to the Lebesgue measure on $\RR^{2p}$. The corresponding invariant density is given by
\begin{align}
p_{*}(\btheta,\bv)\propto\exp\Big\{-f(\btheta) - \frac1{2u}\|\bv\|_2^2\Big\}
,\qquad \btheta\in\RR^p,\ \bv\in\RR^p.
\end{align}
This means that under the invariant distribution, the components $\bL$ and $\bV$ are independent, $\bL$  is
distributed according to the target $\pi$, whereas $\bV/\sqrt{u}$ is a standard Gaussian vector. Therefore, one can use this process for solving the problem of sampling
from $\pi$. As discussed above, the quality of the resulting sampler will depend on two key properties of the process: rate of mixing and smoothness of sample paths. The rate
of mixing of  kinetic diffusions has been recently studied by
\cite{Eberle} under conditions that are more general than strong convexity of $f$. In strongly convex case, a more tractable result has been obtained by
\cite{Cheng2}. It establishes that for $\gamma = 2$ and $u=1/M$, the mixing rate in the Wasserstein distance is $e^{-(m/2M)t}$; see Theorem 5 in \citep{Cheng2}. On the other hand, sample paths of
the process $\{\bL\}$ defined in \eqref{kinetic} are smooth of order $1+\alpha$, for every $\alpha\in[0,1/2[$. Combining these two properties, \citep{Cheng2} prove that a suitable discretization of \eqref{kinetic} leads to a sampler that achieves an error smaller than $\varepsilon$ in a number of iterations $K$ satisfying $K = O((p/\varepsilon^2)^{1/2}\log(p/\varepsilon))$.

It follows from the discussion of previous paragraphs that the kinetic LMC based on \eqref{kinetic} converges faster than the standard LMC based on \eqref{langevin}. Furthermore, this improved rate of convergence is mainly due to the higher smoothness of sample paths of the underlying Markov process. The main purpose of the present work is to
pursue the investigation of the kinetic Langevin Monte Carlo (KLMC) initiated in
\citep{Cheng2} by addressing the following questions:
\vspace{-10pt}
\begin{itemize}\itemsep=1pt
\item[\bf Q1.] What is the rate of mixing of the continuous-time kinetic
Langevin diffusion for general values of the parameters $u$ and $\gamma$?
\item[\bf Q2.] Is it possible to improve the rate of convergence of the KLMC by
optimizing it over the choice of $u$, $\gamma$ and the step-size ?
\item[\bf Q3.] If the function $f$ happens to have a Lipschitz-continuous Hessian, is it
possible to devise a discretization that takes advantage of this additional
smoothness and leads to improved rates of convergence?
\end{itemize}
\vspace{-10pt}
The rest of the paper is devoted to answering these questions. The rate of mixing
for the  continuous-time process is discussed in \Cref{sec:contraction}.  In a
nutshell, we show that if $\gamma\ge \sqrt{(M+m)u}$, then the rate of mixing is of order
$e^{-(um/\gamma)t}$. Non-asymptotic guarantees for the KLMC algorithm are stated and
discussed in \Cref{sec:KLMC}. They are in the same spirit as those established in
\citep{Cheng2}, but have an improved dependence on the condition number, the ratio of
the Lipschitz constant $M$ and the strong convexity constant $m$. Our result has also
improved constants and is much less sensitive to the choice of the initial distribution.
These improvements are achieved thanks to a more careful analysis of the discretization
error of the Langevin process. Finally, we present in \Cref{sec:2ndorderKLMC} a new
discretization, termed second-order KLMC, of the kinetic Langevin diffusion that exploits
the knowledge of the Hessian of $f$. Its error, measured in the Wasserstein distance
$W_2$ is shown  to be bounded by $\varepsilon$ for a number of iterations that scales
as $(p/\varepsilon)^{1/2}$. Thus, we get an improvement of order $(1/\varepsilon)^{1/2}$
over the first-order KLMC algorithm.

\section{Mixing rate of the kinetic Langevin diffusion} \label{sec:contraction}

Let us denote by $\bfP_t^{\bL}$ the transition probability at time $t$ of the kinetic
diffusion $\bL$ defined by \eqref{kinetic}. This means that $\bfP_t^{\bL}$ is a
Markov kernel given by $\bfP_t^{\bL}((\bx,\bv), B) = \bfP(\bL_t\in B|\bV_0=\bv,\bL_0 =
\bx)$, for every
$\bv,\bx\in\RR^p$ and any Borel set $B\subset\RR^p$. For any probability distribution $\mu$
on $\RR^p\times\RR^p$, we denote $\mu\bfP_t^{\bL}$ the (unconditional) distribution of the
random variable  $\bL_t$ when the starting distribution of the process $(\bV,\bL)$ is
$\mu$ (\textit{i.e.}, when $(\bV,\bL_0)\sim \mu$).

Since the process $(\bV,\bL)$ is ergodic, whatever the initial distribution, for large
values of $t$ the distribution of $\bL_t$ is close to the invariant distribution. We want
to quantify how fast does this convergence occur. Furthermore, we are interested in
a nonasymptotic result in the Wasserstein-Kantorovich distance $W_2$, valid for a large
set of possible values $(\gamma,u)$.

A first observation is that, without loss of generality, we can focus our attention to
the case $u=1$. This is made formal in the next lemma.

\begin{lemma}\label{lem:1}
Let $(\bV,\bL)$ be the kinetic  Langevin diffusion defined by \eqref{kinetic}.
The modified process $(\bar\bV_t,\bar\bL_t) = (u^{-1/2}\bV_{t/\sqrt{u}},\bL_{t/\sqrt{u}})$
is an kinetic Langevin diffusion as well with associated parameters $\bar\gamma =
\gamma/\sqrt{u}$  and $\bar u=1$.
\end{lemma}

The proof of this result is straightforward and therefore is omitted.
Note that it shows that the parameter $u$ merely represents a time scale (the speed of
running over the path of the process $\bL$).  Therefore, in the rest of this paper,
we will consider the parameter $u$ to be equal to 1.

\begin{theorem}\label{th:1}
Assume that the function $f$ is twice differentiable with a Hessian matrix
$\nabla^2f$ satisfying $m\bfI_p\preceq \nabla^2f(\bx)\preceq M\bfI_p$ for every
$\bx\in\RR^p$.
Let $\mu_1,\mu_2$ and $\mu_2'$ be three probability measures on $\RR^p$. Let
us define the product measures $\mu = \mu_1\otimes\mu_2$ and
$\mu' = \mu_1\otimes\mu_2'$. For every $\gamma,t>0$, there exist numbers
$\alpha\le \sqrt{2}/\gamma$ and $\beta\ge \{m\wedge (\gamma^2-M) \}/\gamma$
such that
\begin{align}\label{exp1}
W_2(\mu\bfP_t^{\bL},\mu'\bfP_t^{\bL})
	&\le \alpha
	e^{-\beta\,t}
	W_2(\mu,\mu').
\end{align}
More precisely, for every $v\in [0,\gamma/2[$, we have\footnote{One can observe that \eqref{exp1} can be deduced from
\eqref{exp2} by taking $v=0$.}
\begin{align}\label{exp2}
W_2(\mu\bfP_t^{\bL},\mu'\bfP_t^{\bL})
	&\le \frac{\sqrt{2((\gamma-v)^2+v^2)}}{\gamma-2v}
	\exp\bigg\{\frac{(v^2- m)\vee(M-(\gamma-v)^2)}{\gamma-2v}\,t\bigg\}
	W_2(\mu,\mu').
\end{align}
\end{theorem}

The proof of this result is postponed to \Cref{sec:proof1}. Here, we will discuss
some consequences of it and present the main ingredient of the proof. First of all, note that
this result implies that for $\gamma^2 > 2\vee M$, the operator $\bfP_t^{\bL}$ is a
contraction. The rate of this contraction is characterized by the parameter $\beta$.
If we optimize the exponent in \eqref{exp2} with respect to $v$, we get the optimal
rates of contraction reported in \Cref{tab:1}.

If we consider the case $\gamma = 2\sqrt{Mu}= 2\sqrt{M}$ previously studied in \citep{Cheng2},
then the best rate of contraction provided by \eqref{exp2} corresponds to $v = \sqrt{M}-\sqrt{M-m}$,
and the upper bound of \Cref{th:1} reads as
\begin{align}\label{contr:1}
W_2(\mu\bfP_t^{\bL},\mu'\bfP_t^{\bL})
	&\le \bigg(\frac{2M-m}{M-m}\bigg)^{1/2}\exp\big\{-\big(\sqrt{M}- \sqrt{M-m}\big)\,t\big\}W_2(\mu,\mu').
\end{align}
One can check that the constant $\sqrt{M}- \sqrt{M-m}$ that we
obtain within the exponential is optimal, in the sense that one gets exactly this
constant in the case where $f$ is the bivariate quadratic function
$f(x_1,x_2) = (m/2)x_1^2+(M/2)x_2^2$. This constant is slightly better than
the one obtained in \citep[Lemma 8]{Cheng2} for the particular choice of the time scale
$u=1/M$. Indeed, if we rewrite the two results in the common time-scale $u=1$,
\citep[Lemma 8]{Cheng2} provides the contraction rate $\beta =m/(2\sqrt{M} )$, which
is smaller than (but asymptotically equivalent to)
$\sqrt{M}- \sqrt{M-m}$.

Another relevant consequence is obtained by instantiating \eqref{exp1} to
the case $\gamma \ge \sqrt{M+m}$. This leads to the bound
\begin{align}\label{contr:2}
\gamma \ge \sqrt{M+m}\quad
\Longrightarrow
\quad
W_2(\mu\bfP_t^{\bL},\mu'\bfP_t^{\bL})
	&\le \sqrt{2}\exp\big\{-({{m}/{\gamma})}
	\,t\big\}W_2(\mu,\mu').
\end{align}
This result is interesting since it allows to optimize the argument of the exponent
with respect to $\gamma$ for  fixed $t$. The corresponding optimized constant
is $m/\sqrt{M+m}$, which improves on the constant obtained in \eqref{contr:1} for
$\gamma = 2\sqrt{M}$. When $M/m$ becomes large, the improvement
factor gets close to 2.

\begin{table}[ht]
\begin{tabular}{c|c|c|c|c}
\toprule
$\gamma^2\in$ & $]0,M]$ & $]M, m+M]$ &$[m+M,3m+M[$ & $[3m+M,+\infty[$\\
\midrule
rate of contraction, $\beta$ &  NA & $\ds\frac{\gamma^2-M}{\gamma}$
& $\ds\frac{\gamma}{2} -\frac{M-m}{2\sqrt{2(m+M)-\gamma^2}}$ &
$\ds\frac{\gamma-\sqrt{\gamma^2-4m}}{2}$\\
\midrule
Obtained by Thm.~\ref{th:1}  with & - & $v=0$ & $v= \frac{\gamma-\sqrt{2(m+M)-\gamma^2}}{2}$
& $v=\frac{\gamma-\sqrt{\gamma^2-4m}}{2}$ \\
\bottomrule
\end{tabular}
\caption{The rates of contraction of the distribution of the kinetic Langevin
diffusion $\bL_t$ for $u=1$ and varying $\gamma$. The reported values are obtained by
optimizing the bound in \Cref{th:1} with respect to $v$. In the overdamped case
$\gamma^2\ge 3m+M$, the obtained rates coincide with those that can be
directly computed for quadratic functions $f$ and, therefore, are optimal.}
\label{tab:1}
\end{table}

We now describe the main steps of the proof of \Cref{th:1}.
The main idea is to consider along with the process $(\bV,\bL)$, another process
$(\bV',\bL')$ that satisfies the same SDE
\eqref{kinetic} as $(\bV,\bL)$, with the same Brownian motion but with
different initial conditions. One easily checks that
\begin{align}\label{ODE}
    d
		\begin{bmatrix}
		\bV_t-\bV'_t\\
		\bL_t-\bL'_t
		\end{bmatrix}
		&=
		\begin{bmatrix}
		-(\gamma (\bV_t-\bV'_t) + \nabla f(\bL_t)-\nabla f(\bL'_t))\\
		\bV_t-\bV'_t
		\end{bmatrix}
		\,dt\qquad t\ge 0.
\end{align}
Using the mean value theorem, we infer that for a suitable symmetric matrix
$\bfH_t$, we have $\nabla f(\bL_t)-\nabla f(\bL'_t) = \bfH_t(\bL_t-\bL'_t)$. Furthermore,
$\bfH_t$ being the Hessian of a strongly convex function satisfies $\bfH_t\succeq m\bfI_p$.
Then, \eqref{ODE} can be rewritten as
\begin{align}\label{ODE2}
    \frac{d}{dt}
		\begin{bmatrix}
		\bV_t-\bV'_t\\
		\bL_t-\bL'_t
		\end{bmatrix}
		&=
		\begin{bmatrix}
		-\gamma\bfI_p  & -\bfH_t \\
		\bfI_p & \mathbf 0_{p\times p}
		\end{bmatrix}
				\begin{bmatrix}
		\bV_t-\bV'_t\\
		\bL_t-\bL'_t
		\end{bmatrix}
		\qquad t\ge 0.
\end{align}
In a small neighborhood of any fixed time instance $t_0$, \eqref{ODE2} is close to a
linear differential equation with the associated matrix
$$
\bfM(t_0) = 		\begin{bmatrix}
		-\gamma\bfI_p  & -\bfH_{t_0} \\
		\bfI_p & \mathbf 0_{p\times p}
		\end{bmatrix}.
$$
It is well-known that the solution of such a differential equation will tend to zero
%\footnote{To avoid unnecessary technicalities, we
%assume in this discussion that the matrix $\bfM$ is
%diagonalizable.}
if and only if the real parts of all the eigenvalues of $\bfM(t_0)$ are negative.
The matrix $\bfM(t_0)$ is not symmetric; it is in most cases diagonalizable but
its eigenvectors generally depend on $t_0$. To circumvent this difficulty, we
determine the transformations diagonalizing the surrogate matrix
$$
\bfM = 		\begin{bmatrix}
		-\gamma\bfI_p  & -v^2\bfI_p \\
		\bfI_p & \mathbf 0_{p\times p}
		\end{bmatrix}, \qquad\text{for some $v\in[0,\gamma/2[$}.
$$
This yields an invertible  matrix $\bfP$ such that $\bfP^{-1}\bfM\bfP$ is diagonal.
We can thus rewrite \eqref{ODE2} in the form
\begin{align}\label{ODE3}
    \frac{d}{dt}
		\bfP^{-1}\begin{bmatrix}
		\bV_t-\bV'_t\\
		\bL_t-\bL'_t
		\end{bmatrix}
		&= \{\bfP^{-1}\bfM(t)\bfP\}\bfP^{-1}
		\begin{bmatrix}
		\bV_t-\bV'_t\\
		\bL_t-\bL'_t
		\end{bmatrix}
		\qquad t\ge 0.
\end{align}
Interestingly, we prove that the quadratic form associated with
the  matrix $\bfP^{-1}\bfM(t)\bfP$ is negative definite and this provides the desired
result. Furthermore, we use this same matrix $\bfP$ for analyzing the discretized
version of the kinetic Langevin diffusion and proving the main result of the next section.

\section{Error bound for the KLMC in Wasserstein distance} \label{sec:KLMC}

Let us start this section by recalling the KLMC algorithm, the sampler
derived from  a suitable time-discretization of the kinetic diffusion,
introduced by \cite{Cheng2}. Let us define the sequence of functions $\psi_k$ by
$\psi_0(t) = e^{-\gamma t}$ and $\psi_{k+1} (t)= \int_0^t \psi_k(s)\,ds$.
Recall that $f$ is assumed twice differentiable and,
without loss of generality, the parameter $u$ is assumed to be equal to one.
The discretization involves a step-size $h>0$ and is defined by the following
recursion:
\begin{align}\label{KLMC}
\begin{bmatrix}
	\bv_{k+1}\\[4pt]
	\bvartheta_{k+1}
\end{bmatrix}
&=
\begin{bmatrix}
	\psi_0(h)\bv_k-\psi_1(h)\nabla f(\bvartheta_k)\\[4pt]
	\bvartheta_{k} + \psi_1(h)\bv_k -\psi_2(h)\nabla f(\bvartheta_k)
\end{bmatrix}
+ \sqrt{2\gamma }
\begin{bmatrix}
	\bxi_{k+1}\\[4pt]
	\bxi_{k+1}'
\end{bmatrix},	
\end{align}
where $(\bxi_{k+1},\bxi'_{k+1})$ is a $2p$-dimensional centered Gaussian vector
satisfying the following conditions:
\vspace{-18pt}

\begin{itemize}\itemsep=0pt
\item $(\bxi_j,\bxi'_j)$'s are iid and independent of the initial condition $(\bv_0,\bvartheta_0)$,
\item for any fixed $j$, the random vectors $\big((\bxi_{j})_1,(\bxi'_j)_1\big)$,
$\big((\bxi_{j})_2,(\bxi'_j)_2\big)$, $\ldots$, $\big((\bxi_{j})_p,(\bxi'_j)_p\big)$
are iid with the covariance matrix
$$
\mathbf C = \int_0^h [\psi_0(t) \ \psi_1(t)]^\top [\psi_0(t) \ \psi_1(t)]\,
dt.
$$
\end{itemize}
\vspace{-8pt}

This recursion may appear surprizing, but one can check that it is obtained
by first replacing in \eqref{kinetic}, on each time interval $t\in[kh,(k+1)h]$,
the gradient $\nabla f(\bL_t)$ by $\nabla f(\bL_{kh})$, by renaming $(\bV_{kh},\bL_{kh})$
into $(\bv_k,\bvartheta_k)$ and by explicitly solving the obtained linear SDE (which leads to an 
Ornstein-Uhlenbeck process). To the best of our knowledge, the algorithm \eqref{KLMC}, that 
we will refer to as KLMC, has been first proposed by \cite{Cheng2}. The next result 
characterizes its approximation properties.

\begin{theorem}\label{th:3}
Assume that the function $f$ is twice differentiable with a Hessian matrix
$\nabla^2f$ satisfying $m\bfI_p\preceq \nabla^2f(\bx)\preceq M\bfI_p$ for every
$\bx\in\RR^p$.
In addition, let the initial condition of the KLMC algorithm be drawn from the product
distribution $\mu = \mathcal N(\mathbf 0_p,\bfI_p)\otimes \nu_0$.  For every $\gamma\ge
\sqrt{m+M}$ and $h\le m/(4\gamma M)$, the distribution $\nu_k$ of the $k$th iterate
$\bvartheta_k$ of the KLMC algorithm \eqref{KLMC} satisfies
\begin{align}
W_2(\nu_k,\pi)
&\le \sqrt{2}\Big(1-\frac{0.75mh}{\gamma}\Big)^kW_2(\nu_0,\pi) +
\frac{Mh\sqrt{2p}}{m}.
\end{align}
\end{theorem}

The proof of this theorem, postponed to \Cref{sec:proofKLMC}, is inspired by the proof
in \citep{Cheng2}, but with a better control of the discretization error.
This allows us to achieve the following improvements as compared to aforementioned paper:
\vspace{-10pt}
\begin{itemize}\itemsep=0pt
\item The second term in the upper bound provided by \Cref{th:3} scales linearly
as a function of the condition number $\varkappa\triangleq M/m$, whereas the corresponding term in \citep{Cheng2} scales as $\varkappa^{3/2}$.

\item The impact of the initial distribution $\nu_0$ on the overall error of sampling
appears only in the first term, which is multiplied by a sequence that has an exponential
decay in $k$.  As a consequence, if we denote by
$K$ the number of iterations sufficient for the error to be smaller than a prescribed
level $\varepsilon$, our result leads to an expression of $K$ in which $W_2(\nu_0,\pi)$
is within a logarithm. Recall that the expression of $K$ in \citep[Theorem 1]{Cheng2}
scales linearly in $W_2(\nu_0,\pi)$.

\item The numerical constants of \Cref{th:3} are much smaller than those
of the corresponding result in \citep{Cheng2}.
\end{itemize}

In order to ease the comparison of our result to \citep[Theorem 1]{Cheng2},
let us apply \Cref{th:3} to
\begin{align}\label{hKLMC}
h = \frac{m}{4M\sqrt{m+M}}\bigwedge \frac{0.94\varepsilon}{\varkappa\sqrt{2p}}
\end{align}
and  $\gamma = \sqrt{m+M}$, which corresponds to the tightest upper bound
furnished by our theorem. Note that in \citep{Cheng2} it is implicitly assumed that
$p/\varepsilon^2$ is large enough so that the second term in the minimum appearing
in \eqref{hKLMC} is smaller than the first term.
From \eqref{hKLMC} we obtain that\footnote{This value of $K$ is obtained by
choosing $h$ and $K$ so that the second term in the upper bound of \Cref{th:3} is
equal to $(1-\sqrt{2}/24)\epsilon$ whereas the first term is smaller than
$(\sqrt{2}/24)\epsilon$.}
\begin{align}\label{KKLMC}
K_{\rm KLMC}  &\ge \frac{\sqrt{m+M}}{0.75m} \bigg(\frac{4M\sqrt{m+M}}{m}\bigvee
\frac{\varkappa\sqrt{2p}}{0.94\varepsilon}\bigg) \log\bigg(\frac{24W_2(\nu_0,\pi)}{\varepsilon}\bigg)
\end{align}
iterations are sufficient for having $W_2(\nu_K,\pi)\le \varepsilon$. After some
simplifications, we get
\begin{align}
K_{\rm KLMC}  &\ge 3\varkappa^{3/2} \Big\{(16\varkappa)\bigvee
\frac{p}{m\varepsilon^2}\Big\}^{1/2} \log\bigg(\frac{24W_2(\nu_0,\pi)}{\varepsilon}\bigg)
\label{Keps}
\end{align}
Remind that
the corresponding result in \cite{Cheng2} requires $K$ to satisfy\footnote{This lower bound on $K$ is obtained by replacing
$\mathcal D^2\triangleq\|\btheta_0-\btheta^*\|_2$ by 0 in \citep[Theorem 1]{Cheng2}.}
\begin{align}
K\ge 52\varkappa^2\,\Big\{\frac{p}{m\varepsilon^2}\Big\}^{1/2} \log\bigg(\frac{24W_2(\nu_0,\pi)}{\varepsilon}\bigg).
\end{align}
Thus, the improvement in terms of the number of iterations
we obtain is at least by a factor $17\sqrt{\varkappa}$, whenever 
$\kappa\le p/(16m\varepsilon^2)$.

It is also helpful to compare the obtained result \eqref{Keps} to the
analogous result for the highly overdamped Langevin diffusion \citep{Durmus2}.
Using \citep[Eq. (22)]{Durmus4}, one can check that this is enough to choose an integer
\begin{align}
K_{\rm LMC} \ge
2\varkappa \Big\{ 1 \bigvee \frac{2.18 p}{m\varepsilon^2}\Big\}
\log\bigg(\frac{24W_2(\nu_0,\pi)}{\varepsilon}\bigg),
\label{Keps1}
\end{align}
such that $K_{\rm LMC}$ iterations of the LMC algorithm are sufficient to
arrive at an error bounded by $\varepsilon$. Comparing \eqref{Keps} and
\eqref{Keps1}, we see that the KLMC is preferable to the LMC when
$p/(m\varepsilon^2)$ is large as compared to the condition number $\varkappa$.
This is typically the case when the dimensionality is high or a high precision
approximation is required. The order of preference is reversed when
the condition number $\varkappa$ is large as compared to $p/(m\varepsilon^2)$.
Such  a situation corresponds to settings where the target log-density
$f$ is nearly flat ($m$ is small) or has a gradient that may increase very fast
($M$ is large). As an important conclusion, we can note that none of these
two methods is superior to the other in general. The plot in \Cref{fig:1}
illustrates this fact by showing in gray the regions where LMC outperforms KLMC.
\begin{figure}
\includegraphics[width=0.7\textwidth]{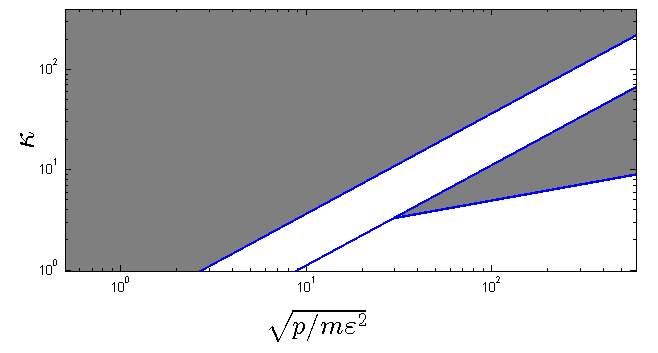}
\caption{This plot represents in the plane defined by coordinates
$(\sqrt{p/m\varepsilon^2},\varkappa)$ the regions where LMC leads
to smaller error than the KLMC (in gray). Please note that the axes are in
logarithmic scale.}
\label{fig:1}
\end{figure}

\section{Second-order KLMC and a bound on its error} \label{sec:2ndorderKLMC}

In this section, we propose another discretization of the kinetic Langevin process,
which is applicable when the function $f$ is twice differentiable. We show below
that this new discretization leads to a provably better sampling error under the
condition that the Hessian matrix of $f$ is Lipschitz-continuous with respect to
the spectral norm.
At any iteration $k\in\mathbb N$, we define $\bfH_k = \nabla^2 f(\bvartheta_k)$ and
\begin{align}\label{KLMC2}
\begin{bmatrix}
	\bv_{k+1}\\[4pt]
	\bvartheta_{k+1}
\end{bmatrix}
&=
\begin{bmatrix}
	\psi_0(h)\bv_k-\psi_1(h)\nabla f(\bvartheta_k) -
	\varphi_2(h)\bfH_k\bv_k\\[4pt]
	\bvartheta_{k} + \psi_1(h)\bv_k -\psi_2(h)\nabla f(\bvartheta_k)
	-\varphi_3(h)\bfH_k\bv_k
\end{bmatrix}
+ \sqrt{2\gamma }
\begin{bmatrix}
	\bxi_{k+1}^{(1)} -\bfH_k\bxi_{k+1}^{(3)}\\[4pt]
	\bxi_{k+1}^{(2)} -\bfH_k\bxi_{k+1}^{(4)}
\end{bmatrix},
\end{align}
where

\vspace{-12pt}
\begin{itemize}\itemsep=1pt
\item $\psi_0,\psi_1,\psi_2$ are defined as in the beginning of the previous section,
\item $\varphi_{k+1}(t) = \int_0^t e^{-\gamma (t-s)} \psi_k(s)\,ds$ for every $t>0$,
\item the $4p$ dimensional random vectors $(\bxi_{k+1}^{(1)},\bxi_{k+1}^{(2)},
\bxi_{k+1}^{(3)},\bxi_{k+1}^{(4)})$ are iid Gaussian with zero mean,
\item for any fixed $j$, the $4$-dimensional random vectors
$\big([(\bxi_{j}^{(1)})_1,(\bxi_{j}^{(2)})_1,(\bxi_j^{(3)})_1,(\bxi_j^{(4)})_1]$,
$\ldots$,\\
$[(\bxi_{j}^{(1)})_p,(\bxi_{j}^{(2)})_p,(\bxi_j^{(3)})_p,(\bxi_j^{(4)})_p]\big)$
are iid with the covariance matrix
\begin{align}
\bar\bfC =
\int_0^h [\psi_0(t);\, \psi_1(t);\, \varphi_2(t);\, \varphi_3(t)]^\top
[\psi_0(t);\, \psi_1(t);\, \varphi_2(t);\, \varphi_3(t)]\,dt.
\end{align}
\end{itemize}
This definition is somewhat complicated, but it follows from an application
of the second-order Taylor approximation to the drift term of the kinetic
Langevin diffusion\footnote{For more detailed explanations, see \Cref{ssec:expKLMC2}}.
At this stage, one can note that if the Hessian $\bfH_k$
is zero, then the update rule \eqref{KLMC2} boils down to the update rule of
the KLMC algorithm in \eqref{KLMC}. Iterating the update rule \eqref{KLMC2} we
get a random variable that will be henceforth called KLMC2 or second-order
kinetic Langevin Monte-Carlo algorithm.

\begin{theorem}\label{th:4}
Assume that, for some constants $m,M,M_2>0$, the function $f$ is
$m$-strongly convex, its gradient is $M$-Lipschitz, and its Hessian is $M_2$-Lipschitz for the spectral norm.
In addition, let the initial condition of the second-order KLMC algorithm be drawn
from the product distribution $\mu = \mathcal N(\mathbf 0_p,\bfI_p)\otimes \nu_0$.
For every 
$$
\gamma\ge \sqrt{m+M}\quad \text{and}\quad h\le \frac{m}{5\gamma M}\wedge
\frac{m}{4\sqrt{5p}\,M_2}, 
$$
the distribution
$\nu_k^{\rm KLMC2}$  of the $k$th iterate $\bvartheta_k^{\rm KLMC2}$ of the
second-order KLMC algorithm \eqref{KLMC2} satisfies\footnote{One can see
from the proof that $e^{-p/2}$ in this inequality can be replaced
by the smaller quantity $e^{-\frac{m^2}{160 M_2^2h^2}}$.
}
\begin{align}
		W_2(\nu_k^{\rm KLMC2},\pi)
		&\le \sqrt{2}\Big(1-\frac{mh}{4\gamma}\Big)^{k}W_2(\nu_0,\pi)		
		+\frac{2h^2{M_2p}}{m}+\frac{h^2 M\sqrt{2Mp}}{{m}}
		+ \frac{8M}{m}\, h e^{-p/2} .\label{s5:1}
\end{align}
\end{theorem}

Several important consequences can be drawn from this result. First, 
the value of the parameter $\gamma$ minimizing the right hand side is its smallest
possible value $\gamma = \sqrt{m+M}$. Second, one can note that the last term of 
the obtained upper bound is independent of dimension $p$ and decreases exponentially
fast in $1/h$. This term is in most cases negligible with respect to the other terms
involved in the upper bound. In particular, we deduce from this result that if 
the Lipschitz constants $M$ and $M_2$ are bounded and the strong convexity constant 
$m$ is bounded  away from zero, then the KLMC2 algorithm achieves the precision level 
$\varepsilon$ after $K_\varepsilon$ iterations, with $K_\varepsilon$ being of order 
$\sqrt{p/\varepsilon}$, up to a logarithmic factor. Finally, if we neglect the last 
term in the upper bound of \Cref{th:4}, and choose the parameters $h$ and $k$ 
so that the other terms are equal to $\varepsilon/\sqrt{4m}$, we get that
the number of iteration $K_{\varepsilon}$ to achieve an error $\varepsilon/\sqrt{m}$ 
scales, up to a logarithmic factor, as  $\sqrt{M}/(mh_\varepsilon) = 
\sqrt{p}\,\varkappa_2^2 + \sqrt{p/\varepsilon}\,\varkappa_2^{5/4}$, where 
$\varkappa_2 = (M_2^{2/3} + Mp^{-1/3})/m$ is a version of the condition number 
taking into account the Hessian-Lipschitz assumption.

It is interesting to compare this result to the convergence result for the LMCO
algorithm established in \citep{DalKar17}. We can note that the number of iterations
that are sufficient for the KLMC2 to achieve the error $\varepsilon$ is much smaller
than the corresponding number for the LMCO: $\sqrt{p/\varepsilon}$ versus
$p/\varepsilon$. In addition, the KLMC2 algorithm does not need to compute matrix
exponentials neither to do matrix inversion. The most costly operations are that
of computing the products of the $p\times p$ Hessian and the vectors $\bv_k$, 
$\bxi_{k+1}^{3}$ and $\bxi_{k+1}^{3}$. In most cases, the complexity of these 
computations scales linearly in $p$. 

As a conclusion, to the best of our knowledge, the second-order KLMC algorithm
provides the best known convergence rate $\sqrt{p/\varepsilon}$ for a target density
$\pi$ having a log-density that is concave and Hessian-Lipschitz.

\section{Related work}\label{sec:relwork}

The idea of using the Langevin diffusion (see \citep{Pavliotis14} for an introduction
to this topic) for approximating a random variable drawn from its invariant distribution
is quite old and can be traced back at least to \citep{RobertsTweedie96}. Since then,
many papers focused on analyzing the asymptotic behavior of the Langevin-based methods
under various assumptions, see \citep{Lamberton1,Lamberton2,StramerTweedie99-1,StramerTweedie99-2, douc2004,Pillai2012,
Xifara14,RobertsStramer02,RobertsRosenthal98,Hairer2013} and the references therein.
Convergence to the invariant distribution for Langevin processes is studied
in \citep{Desvillettes,Helffer,Dolbeault}.

Non-asymptotic and computable bounds on the convergence
to equilibrium of the kinetic Langevin diffusion have been recently obtained in
\citep{Eberle, 2018Cheng,Cheng2}. While \citep{Cheng2} considers only the convex case,
\citep{Eberle, 2018Cheng} deal also with nonconvexity. On the one hand,
\citep{2018Cheng} provide results only for a fixed value of parameters $(\gamma,u) = (2,1/M)$.
On the other hand, if we instantiate results of \citep{Eberle} to the case of convex functions $f$,
convergence to the invariant density is proved under the condition $\gamma^2\ge 30 Mu$.
This is to be compared to the conditions of \Cref{th:1} that establishes exponential
convergence as soon as $\gamma^2> Mu$.

Nonasymptotic bounds on the precision of the Langevin Monte Carlo under strong convexity
have been established in \citep{Dal14} and then extended and refined in a series of papers
\citep{Durmus2,Bubeck15,DalalyanColt,Cheng1,durmus2017,brosse17a,Durmus4,luu,pmlr-v75-bernton18a}. Very recently, it was proved in \citep{pmlr-v75-dwivedi18a} that applying
a Metropolis-Hastings correction to the LMC leads to improved dependence on the
target precision $\epsilon$ of the number of gradient evaluations. The fact that the discretized version of the kinetic Langevin diffusion may outperform its highly
overdamped counterpart was observed and quantified in \citep{Cheng2}.

Previous work has also studied the precision of Langevin algorithms in the case when the
gradient evaluations are contaminated by some noise \citep{DalalyanColt,DalKar17,Cheng2,Baker2018,pmlr-v80-chatterji18a}
and the relation with stochastic optimization \citep{raginsky17a,zhang17b,2017arXiv170706618X,2017arXiv170706386D}.
There are certainly many other papers related to the present work that are not mentioned in this
section. There is a vast literature on this topic and it will be impossible to quote all
the papers. We believe that the papers cited here and the references therein provide a good
overview of the state of the art.

\section{Conclusion} \label{sec:concl}

In order to summarize the content of the previous sections, let us return, on by one, to the questions raised in the
introduction. First, concerning the mixing properties of the kinetic Langevin diffusion for general values of $u$ and $\gamma$,
we have established that as soon as $\gamma^2 > Mu$, the process mixes exponentially fast with  a rate at least equal to
$\{mu \wedge (\gamma^2-Mu)\}/\gamma$. Therefore, for fixed values of $m$, $M$ and $u$, the nearly fastest rate of mixing
is obtained for $\gamma^2 = (m+M)u$ and is equal to $m/\sqrt{m+M}$.

To answer the second question, we have seen that optimization with respect to $\gamma$ and $u$ leads to improved constants
but does not improve the rate. Indeed, if we use the values of $\gamma$ and $u$ used in \citep{Cheng2} (that is $\gamma= 2$
and $u= 1/M$, which in view of \Cref{lem:1} are equivalent to $\gamma = 2\sqrt{M}$ and $u=1$) lead to a bound on the
number of iterates sufficient to achieve a precision $\varepsilon$ that is of the same order as the optimized one given in
\eqref{KKLMC}. Interestingly, our analysis revealed that not only the numerical constants of the result in \citep{Cheng2}
can be improved, but also the dependence on the condition number $\varkappa = M/m$ can be made better. Indeed, we have
managed to replace the factor $\varkappa^2$ by $\varkappa^{3/2}$. Such an improvement might have important consequences
in generalizing the results to the case of a convex function which is not strongly convex. This line of research will be
explored in a future work. Our bound exhibits also a better dependence on the error of the first step: it is
logarithmic in our result while it was linear in \citep{Cheng2}.

Finally, we have given an affirmative answer to the third question. We have shown that leveraging second-order information
may reduce the number of steps of the algorithm by a factor proportional to $1/\sqrt{\varepsilon}$, where $\varepsilon$ is
the target precision. In order to better situate this improvement in the context of prior work, the table below reports
the order of magnitude of the number of steps\footnote{To ease the comparison, we consider $\varkappa$ as a fixed constant
and do not report the dependence on $\varkappa$ in this table.} of Langevin related algorithms in the strongly convex case:

\begin{center}
\begin{tabular}{c|c|c}
\toprule
1st-order LMC  & 1st-order KLMC  & 2nd-order KLMC \\
\midrule
\citep{Durmus2} &  \citep{Cheng2}  & \Cref{th:4}\\
\citep{DalKar17} & and \Cref{th:3} & \\
\midrule
$ p/\varepsilon$  & $ \sqrt{p}/\varepsilon$ & $\sqrt{p/\varepsilon}$\\
\bottomrule
\end{tabular}
\end{center}

\section{Proof of the mixing rate}\label{sec:proof1}

This section is devoted to proofs of the results stated in \Cref{sec:contraction}.
Let $\bL_0,\bL_0'$ and $\bV_0$ be three $p$-dimensional random vectors defined
on the same probability space such that
\vspace{-12pt}
\begin{itemize}\itemsep=0pt
\item $\bV_0$ is independent of $(\bL_0,\bL_0')$,
\item $\bV_0\sim \mu_1$, whereas $\bL_0\sim \mu_2$ and $\bL'_0\sim\mu_2'$,
\item $W_2^2(\mu_2,\mu_2') = \bfE[\|\bL_0-\bL_0'\|_2^2]$.
\end{itemize}
\vspace{-12pt}

Let $\bW$ be a Brownian motion on the same probability space. We define
$(\bV,\bL)$ and $(\bV',\bL')$ as kinetic Langevin diffusion processes
driven by the same Brownian motion $\bW$ and satisfying the initial condition
$\bV'_0=\bV_0$. From the definition of the Wasserstein distance, it follows
that
$$
W_2^2(\mu\bfP_t^{\bL},\mu'\bfP_t^{\bL}) \le \bfE[\|\bL_t-\bL_t'\|_2^2].
$$
In view of this inequality, it suffices to find an appropriate upper
bound on the right hand side of the last display, in order to prove
\Cref{th:1}. This upper bound is provided below in \Cref{prop:1}.

\begin{proposition}\label{prop:1}
Let $\bV_0,\bL_0$ and $\bL_0'$ be random vectors in $\mathbb{R}^{p}$. Let $(\bV_t,\bL_t)$
and $(\bV'_t,\bL'_t)$ be kinetic Langevin diffusions driven by the same Brownian motion
and starting from $(\bV_0,\bL_0)$ and $(\bV_0,\bL'_0)$, respectively. Let $v$ be an
arbitrary real number from $[0,\gamma/2)$. We have
\begin{align}
\|\bL_t-\bL'_t\|_2
	&\le \frac{\sqrt{2((\gamma-v)^2+v^2)}}{\gamma-2v}
	\exp\bigg\{\frac{(v^2-m)\vee(M-(\gamma-v)^2)}{\gamma-2v}\,t\bigg\}
	\|\bL_0-\bL'_0\|_2,\quad \forall t\ge 0.
\end{align}
\end{proposition}

\begin{remark}
As a consequence, we can see that for $\gamma^2\ge 2(M+m)$ by setting
$$
v = \frac{\gamma-\sqrt{\gamma^2-4m}}{2}\ge \frac{m}{\gamma}.
$$
we arrive at
\begin{equation}
\|\bL_t-\bL'_t\|_2
	\le  \bigg(\frac{2\gamma^2-4m}{\gamma^2-4m}\bigg)^{1/2}
	e^{-v t}\,\|\bL_0-\bL'_0\|_2,\qquad \forall t\ge 0.
\end{equation}
\end{remark}

\begin{proof}
We will use the following short hand notations $\psi_t\triangleq (\bV_t+\lambda_+\bL_t)
-(\bV'_t+\lambda_+ \bL'_t)$ and $z_t\triangleq (-\bV_t-\lambda_-\bL_t)+ \bV'_t+
\lambda_-\bL'_t$, where $\lambda_+$ and $\lambda_-$ are two positive numbers such that
$\lambda_++\lambda_-=\gamma$ and $\lambda_+>\lambda_-$. First note that using Taylor's
theorem with the remainder term in integral form, we get
$$
\nabla f(\bL_t)-\nabla f(\bL'_t) = \bfH_t(\bL_t-\bL'_t)
$$
with $\bfH_t\triangleq\int_0^1\nabla^2 f(\bL_t-x(\bL_t-\bL'_t))dx$. In view of
this formula and the fact that $(\bV,\bL)$ and $(\bV',\bL')$ satisfy the SDE
\eqref{kinetic}, we obtain
\begin{align}
\frac{d}{dt} \psi_t
	&=-\gamma(\bV_t-\bV'_t) - \left(\nabla f(\bL_t)-\nabla f(\bL'_t)\right)
			+\lambda_+(\bV_t-\bV'_t)\\
	&=\frac{(\lambda_+-\gamma)(\lambda_-\psi_t+\lambda_+ z_t)}{\lambda_--\lambda_+}
	-\frac{\bfH_t(\psi_t+z_t)}{\lambda_+-\lambda_-}\\
	&= \frac{(\lambda_-^2\bfI-\bfH_t)\psi_t+(\lambda_-\lambda_+\bfI-\bfH_t)z_t}
	{\lambda_+-\lambda_-}.
\end{align}
In the above inequalities, we have used that $\lambda_+-\gamma = -\lambda_-$.
Similar computations yield
\begin{align}
\frac{d}{dt} z_t
	&=\gamma(\bV_t-\bV'_t) + \left(\nabla f(\bL_t)-\nabla f(\bL'_t)\right)
			-\lambda_-(\bV_t-\bV'_t)\\
	&=\frac{(\gamma-\lambda_-)(\lambda_-\psi_t+\lambda_+ z_t)}{\lambda_--\lambda_+}
	+\frac{\bfH_t(\psi_t+z_t)}{\lambda_+-\lambda_-}\\
	&= \frac{(\bfH_t-\lambda_-\lambda_+\bfI)\psi_t+(\bfH_t-\lambda_+^2\bfI)z_t}
	{\lambda_+-\lambda_-}.
\end{align}
From these equations, we deduce that
\begin{align}
    \frac{d}{dt}
		\left\|
		\begin{bmatrix}
     \psi_t  \\
      z_t
		\end{bmatrix}
		\right\|_2^2
				& = 2\psi_t^\top\frac{d\psi_t}{dt} + 2z_t^\top\frac{dz_t}{dt}\\
				&=\frac{2}{\lambda_+-\lambda_-}
				\Big\{ \psi_t^\top(\lambda_-^2\bfI-\bfH_t)\psi_t +
					z_t^\top(\bfH_t-\lambda_+^2\bfI)z_t\Big\}	\\			
				&\le\frac{2}{\lambda_+-\lambda_-}
				\Big\{ (\lambda_-^2-m)\|\psi_t\|_2^2 +
					(M-\lambda_+^2)\|z_t\|_2^2\Big\}\\
				&\le \frac{2\{(\lambda_-^2-m)\vee(M-\lambda_+^2)\}}{\lambda_+-\lambda_-}
						\left\|
						\begin{bmatrix}
						\psi_t  \\
						z_t
						\end{bmatrix}
						\right\|_2^2.
\end{align}
An application of Gronwall's inequality yields
$$
\left\|
	\begin{bmatrix}
     \psi_t  \\
      z_t
	\end{bmatrix}\right\|_2\le \exp\left\{\frac{(\lambda_-^2-m)\vee(M-\lambda_+^2)}
	{\lambda_+-\lambda_-}\,t\right\}\left\|\begin{bmatrix}
     \psi_0 \\
      z_0
\end{bmatrix}\right\|_2,\qquad \forall t\ge 0.
$$
Since $\bV_0=\bV'_0$ and $\bL_t-\bL'_t = (\psi_t+z_t)/(\lambda_+-\lambda_-)$, we get
\begin{align}
\|\bL_t-\bL'_t\|_2
	&\le \frac{\sqrt{2}}{\lambda_+-\lambda_-}\left\|
	\begin{bmatrix}
     \psi_t  \\
      z_t
	\end{bmatrix}\right\|_2\\
	&\le \frac{\sqrt{2(\lambda_+^2+\lambda_-^2)}}{\lambda_+-\lambda_-}
	\exp\bigg\{\frac{(\lambda_-^2-m)\vee(M-\lambda_+^2)}
	{\lambda_+-\lambda_-}\,t\bigg\}
	\|\bL_0-\bL'_0\|_2,\qquad \forall t\ge 0,
\end{align}
and the claim of the proposition follows.
\end{proof}

\section{Proof of the convergence of the first-order KLMC}\label{sec:proofKLMC}

This section contains the complete proof of \Cref{th:3}.
%We first  use the triangle inequality
%\begin{align}\label{s8:1}
%W_2(\nu_k,\pi) \le W_2(\nu_k,\mu\bfP^{\bL}_{kh}) +W_2(\mu\bfP^{\bL}_{kh},\pi),
%\end{align}
%where we remind that $\mu = \mathcal N(\mathbf 0_p,\bfI_p)\otimes\nu_0$ is the
%initial distribution of the sequence $\{(\bv_k,\bvartheta_k)\}$ and $\mu\bfP^{\bL}_{kh}$
%is the distribution of the kinetic Langevin process $\bL$ at time instant $kh$ when the
%initial condition of this process is drawn from $\mu$. It follows from \eqref{contr:2}
%that
%\begin{align}
%W_2(\mu\bfP_{kh}^{\bL},\pi)&\le \sqrt{2}\,e^{-mkh/\gamma}W_2(\nu_0,\pi).
%\end{align}
%Combining with \eqref{s8:1}, we get
%\begin{align}\label{s8:2}
%W_2(\nu_k,\pi) \le W_2(\nu_k,\mu\bfP^{\bL}_{kh}) + \sqrt{2}\,e^{-mkh/\gamma}W_2(\nu_0,\pi).
%\end{align}
We first  write
\begin{align}\label{s8:1}
W_2(\nu_k,\pi) = W_2(\nu_k,\mu^*\bfP^{\bL}_{kh}),
\end{align}
where $\mu^* = \mathcal N(\mathbf 0_p,\bfI_p)\otimes\pi$
and $\mu^*\bfP^{\bL}_{kh}$ is the distribution\footnote{In other words, $\mu^*\bfP^{\bL}_{kh}$ is the first marginal
of the distribution $\mu^*\bfP^{(\bL,\bV)}_{kh}$, the last notation being standard in the theory of Markov processes.}
of the kinetic Langevin process $\bL$ at
time instant $kh$ when the initial condition of this process is drawn from $\mu^*$.
In order to upper bound the term in the right hand side of the last display, we
introduce the discretized version of the kinetic Langevin diffusion:
$(\tilde\bV_0,\tilde\bL_0) \sim \mu$ and for every  $j=0,1,\ldots,k$ and for
every $t\in ]jh,(j+1)h]$,
\begin{align}\label{s8:2}
    \tilde\bV_t&=\tilde\bV_{jh}e^{-\gamma (t-jh)}-\int_{jh}^{t}e^{-\gamma (t-s)}ds
		\nabla f(\tilde\bL_{jh})+\sqrt{2\gamma}\int_{jh}^{t}e^{-\gamma (t-s)}\,\,d\bW_{jh+s}\\
     \tilde\bL_t&= \tilde\bL_{jh}+\int_{jh}^t\tilde\bV_{jh+s}\,ds.
\end{align}
We stress that $\bW$ in the above formula is the same Brownian motion as the one used
for defining the process $(\bV,\bL)$. Furthermore, we choose $\tilde\bV_0=\bV_0$ and
$(\bL_0,\tilde\bL_0)$ so that
\begin{align}\label{s8:3}
W_2^2(\nu_0,\pi) = \bfE[\|\bL_0-\tilde\bL_0\|_2^2].
\end{align}
The process $(\tilde\bV,\tilde\bL)$ realizes the synchronous coupling between the
sequences $\{(\bv_j,\bvartheta_j);j=0,\ldots,k\}$ and
$\{(\bV_{jh},\bL_{jh});j=0,\ldots,k\}$. Indeed, one easily checks by mathematical
induction that  $(\tilde\bV_{jh},\tilde\bL_{jh})$ has exactly the same distribution as
the vector $(\bv_j,\bvartheta_j)$. Therefore, we have
\begin{align}
W_2(\nu_k,\mu^*\bfP^{\bL}_{kh}) \le \big(\bfE[\|\tilde\bL_{kh}-\bL_{kh}\|_2^2]\big)^{1/2}
\triangleq \|\tilde\bL_{kh}-\bL_{kh}\|_{\LL_2}.
\end{align}
Let $\bfP$ be the matrix used in the proof of the contraction in continuous time for $v=0$, that is
$$
\bfP =
\frac1{\gamma}
\begin{bmatrix}
 \mathbf 0_{p\times p} & -\gamma\bfI_p\\
\bfI_p & \bfI_p
\end{bmatrix},\qquad \bfP^{-1}
=\begin{bmatrix}
\bfI_p & \gamma\bfI_p\\
-\bfI_p &  \mathbf 0_{p\times p}
\end{bmatrix}.
$$
We will now evaluate the sequence
\begin{align}
A_k  \triangleq \Bigg\|\bfP^{-1}
\begin{bmatrix}
\tilde\bV_{kh}-\bV_{kh}\\
\tilde\bL_{kh}-\bL_{kh}
\end{bmatrix}
\Bigg\|_{\LL_2}.
\end{align}
The rest of the proof, devoted to upper bounding the last $\mathbb{L}_2$-norm, is done by
mathematical induction. On each time interval $[jh,(j+1)h]$, we introduce
an auxiliary continuous-time kinetic Langevin process $(\bV',\bL')$ such
that $(\bV'_{jh},\bL'_{jh}) = (\tilde\bV_{jh},\tilde\bL_{jh})$ and
\begin{align}\label{kinetic2}
    d
		\begin{bmatrix}
		\bV'_t\\
		\bL'_t
		\end{bmatrix}
		&=
		\begin{bmatrix}
		-(\gamma \bV'_t + \nabla f(\bL'_t))\\
		\bV'_t
		\end{bmatrix}
		\,dt+\sqrt{2\gamma u}
		\begin{bmatrix}
		\bfI_p\\
		\mathbf 0_{p\times p}
		\end{bmatrix}
		\,d\bW_t,\qquad t\in[jh,(j+1)h].
\end{align}
By the triangle inequality, we have
\begin{align}
A_{j+1} &\le \Bigg\|\bfP^{-1}
\begin{bmatrix}
\tilde\bV_{(j+1)h}-\bV'_{(j+1)h}\\
\tilde\bL_{(j+1)h}-\bL'_{(j+1)h}
\end{bmatrix}
\Bigg\|_{\LL_2}+
\Bigg\|\bfP^{-1}
\begin{bmatrix}
\bV'_{(j+1)h}-\bV_{(j+1)h}\\
\bL'_{(j+1)h}-\bL_{(j+1)h}
\end{bmatrix}
\Bigg\|_{\LL_2}\\
&\le \Bigg\|\bfP^{-1}
\begin{bmatrix}
\tilde\bV_{(j+1)h}-\bV'_{(j+1)h}\\
\tilde\bL_{(j+1)h}-\bL'_{(j+1)h}
\end{bmatrix}
\Bigg\|_{\LL_2}+
e^{-mh/\gamma}A_j,\label{s8:5}
\end{align}
where in the last inequality we have used the contraction established in
continuous time. For the first norm in the right hand side of the last display,
we use the fact that the considered processes $(\bV',\bL')$ and
$(\tilde\bV,\tilde\bL)$ have the same value at the time instant $jh$. Therefore,
\begin{align}
\|\tilde\bV_{t}-\bV'_t\|_{\LL_2} &=
\bigg\|\int_{jh}^{t} e^{-\gamma(t-s)}\big(\nabla f(\bL'_{s})-\nabla f(\bL'_{jh})
\big)\,ds\bigg\|_{\LL_2}\\
&\le\int_{jh}^{t} \big\|\nabla f(\bL'_{s})-\nabla f(\bL'_{jh})\big\|_{\LL_2}\,ds\\
&\le M\int_{jh}^{t} \big\|\bL'_{s}-\bL'_{jh}\big\|_{\LL_2}\,ds\\
&\le M\int_{jh}^{t} \int_{jh}^s\big\|\bV'_{u}\big\|_{\LL_2}\,du\,ds\\
& = M\int_{jh}^{t} (t-u)\big\|\bV'_{u}\big\|_{\LL_2}\,du\\
& \le M\int_{jh}^{t} (t-u)\,du \max_{u\in[jh,(j+1)h]}\big\|\bV'_{u}\big\|_{\LL_2}\\
& = \frac{M(t-jh)^2}{2}\max_{u\in[jh,(j+1)h]}\big\|\bV'_{u}\big\|_{\LL_2}
\end{align}
and
\begin{align}
\|\tilde\bL_{(j+1)h}-\bL'_{(j+1)h}\|_2 &=
\bigg\|\int_{jh}^{(j+1)h} (\tilde\bV_{t}-\bV'_t)\,dt\bigg\|_2\\
&\le\int_{jh}^{(j+1)h} \|\tilde\bV_{t}-\bV'_t\|_2\,dt\\
&\le \frac{M}{2}\int_{jh}^{(j+1)h} (t-jh)^2\,dt\max_{u\in[jh,(j+1)h]}\big\|\bV'_{u}\big\|_{\LL_2}\\
&\le \frac{Mh^3}{6}\max_{u\in[jh,(j+1)h]}\big\|\bV'_{u}\big\|_{\LL_2}.
\end{align}
\begin{lemma}
For every $u\in[jh,(j+1)h]$, we have
$$
\|\bV'_u\|_{\LL_2}\le \sqrt{p}+A_j.
$$
\end{lemma}
\begin{proof}
We have
\begin{align}
\|\bV'_u\|_{\LL_2}
	& = \|\bV_u\|_{\LL_2}  + \|\bV'_u-\bV_u\|_{\LL_2} \\
	& = \sqrt{p} + \|[\bfI_p,\ \mathbf 0_p]\bfP\bfP^{-1}[(\bV'_u-\bV_u)^\top,
	(\bL'_u-\bL_u)^\top]\|_{\LL_2}\\
	&\le \sqrt{p}+ \|[\bfI_p,\ \mathbf 0_p]\bfP\|\cdot \|\bfP^{-1}[(\bV'_u-\bV_u)^\top,
	(\bL'_u-\bL_u)^\top]\\
	&\le \sqrt{p}+ \|[\bfI_p,\ \mathbf 0_p]\bfP\|\cdot \|\bfP^{-1}[(\bV'_{jh}-
	\bV_{jh})^\top, (\bL'_{jh}-\bL_{jh})^\top]\|_{\LL_2}\\
	&= \sqrt{p}+ \|[\bfI_p,\ \mathbf 0_p]\bfP\|\cdot A_j.
\end{align}
Recall that
$$
\bfP =
\frac1{\gamma}
\begin{bmatrix}
\mathbf 0_{p\times p} & -\gamma\bfI_p\\
\bfI_p & \bfI_p
\end{bmatrix},
$$
which implies that $\|[\bfI_p,\ \mathbf 0_p]\bfP\| =  1$.
This completes the proof of the lemma.
\end{proof}
From this lemma and previous inequalities, we infer that
\begin{align}
&\Bigg\|\bfP^{-1}
\begin{bmatrix}
\tilde\bV_{(j+1)h}-\bV'_{(j+1)h}\\
\tilde\bL_{(j+1)h}-\bL'_{(j+1)h}
\end{bmatrix}
\Bigg\|_{\LL_2}\\
	&\qquad\le \left\{\left(\|\tilde\bV_{(j+1)h}-\bV'_{(j+1)h}\|_{\LL_2} +
	\gamma\,\|\tilde\bL_{(j+1)h}-\bL'_{(j+1)h}\|_{\LL_2}\right)^2+ \|\tilde\bV_{(j+1)h}-\bV'_{(j+1)h}\|_{\LL_2}^2\right\}^{1/2}\\
	&\qquad\le \left\{\Big(1+\frac{\gamma h}{3}\Big)^2+1\right\}^{1/2} \frac{Mh^2}{2}
	\big(\sqrt{p}+A_j\big).
\end{align}
Choosing $h\le 1/(4\gamma)$, we arrive at
$$
\Bigg\|\bfP^{-1}
\begin{bmatrix}
\tilde\bV_{(j+1)h}-\bV'_{(j+1)h}\\
\tilde\bL_{(j+1)h}-\bL'_{(j+1)h}
\end{bmatrix}
\Bigg\|_{\LL_2}
	\le 0.75\,Mh^2\big(\sqrt{p}+A_j\big).
$$
Combining this inequality and \eqref{s8:5}, for every $h\le m/(4\gamma M)$,
we get
\begin{align}
A_{j+1} &\le 0.75\,Mh^2\big(\sqrt{p}+A_j\big) +
e^{-hm/\gamma} A_j\label{s8:6}\\
& = 0.75 Mh^2\sqrt{p} +
(e^{-hm/\gamma} +0.75\,Mh^2)A_j.\label{s8:7}
\end{align}
Using the inequality $e^{-x}\le 1-x+\frac12 x^2$, we can derive from \eqref{s8:7} that
\begin{align}
A_{j+1} & \le 0.75 Mh^2\sqrt{p} + \bigg(1-\frac{hm}{\gamma} +
\frac{h^2m^2}{2\gamma^2} +0.75\,Mh^2\bigg)A_j\\
& \le 0.75Mh^2\sqrt{p} +\Big(1-\frac{0.75mh}{\gamma}\Big)A_j.
\end{align}
Unfolding this recursive inequality, we arrive at
\begin{align}
A_{k}
& \le \frac{Mh\gamma\sqrt{p}}{m} +
\Big(1-\frac{0.75mh}{\gamma}\Big)^kA_0.
\end{align}
Finally, one easily checks that $A_0 = \gamma W_2(\nu_0,\pi)$ and
\begin{align}
\|\tilde\bL_{kh}-\bL_{kh}\|_{\LL_2}
		&\le \|[ \mathbf 0_{p\times p}\ \bfI_p]\bfP\| A_k = \gamma^{-1}\sqrt{2} A_k.
\end{align}
Putting all these pieces together, we arrive at
\begin{align}
W_2(\nu_k,\pi) &\le \|\tilde\bL_{kh}-\bL_{kh}\|_{\LL_2} \\
		&\le \gamma^{-1}\sqrt{2} A_k\\
		&\le \frac{Mh\sqrt{2p}}{m} + \sqrt{2}
			\Big(1-\frac{0.75mh}{\gamma}\Big)^k(A_0/\gamma)\\
		&= \frac{Mh\sqrt{2p}}{m} + \sqrt{2}
			\Big(1-\frac{0.75mh}{\gamma}\Big)^kW_2(\nu_0,\pi),
\end{align}
and the claim of \Cref{th:3} follows.

\section{Proofs for the second-order discretization of the kinetic Langevin diffusion}

We start this section by providing some explanations on the definition
of the KLMC2 algorithm. We turn then to the proof of \Cref{th:4}.

\subsection{Explanations on the origin of the KLMC2 algorithm}\label{ssec:expKLMC2}

Recall that the kinetic diffusion is given by the equation
\begin{align}\label{Kinetic}
    d
		\begin{bmatrix}
		\bV_t\\
		\bL_t
		\end{bmatrix}
		&=
		\begin{bmatrix}
		-(\gamma \bV_t + \nabla f(\bL_t))\\
		\bV_t
		\end{bmatrix}
		\,dt+\sqrt{2\gamma}
		\begin{bmatrix}
		\bfI_p\\
		\mathbf 0_{p\times p}
		\end{bmatrix}
		\,d\bW_t.
\end{align}
From \eqref{Kinetic}, by integration by parts, we can deduce that
\begin{align}
e^{\gamma t}\bV_t
		&= \bV_0 + \int_0^t e^{\gamma s}\,d\bV_s+\gamma\int_0^t e^{\gamma s}\bV_s\,ds
		\\
		&= \bV_0 - \int_0^t e^{\gamma s}\nabla f(\bL_s)\,ds+\sqrt{2\gamma }
		\int_0^t e^{\gamma s}\,d\bW_s.
\end{align}
Therefore, we have
\begin{align}
\bV_t
		&= e^{-\gamma t}\bV_0 -  \int_0^t e^{-\gamma (t-s)} \nabla f(\bL_s)\,ds+
		\sqrt{2\gamma}\int_0^t e^{-\gamma (t-s)}\,d\bW_s,\label{Vt}\\
\bL_t &= \bL_0 + \int_0^t \bV_s\,ds.		\label{Lt}
\end{align}
\begin{lemma}
For every $\gamma>0$ and $t>0$, we have for any $k,j\in\mathbb{N}$
$$
\varphi_{k+1}(t)=\int_0^t\varphi_k(s)ds,\qquad\varphi_{k+j+1}(t)=\int_0^t \psi_k(s)\psi_j(t-s)ds
$$
\end{lemma}
\begin{proof}Fubini's Theorem and a change of variables yield
\begin{align}
\int_0^t\varphi_k(s)ds&=\int_0^t\int_0^se^{-\gamma(s-r)}\psi_{k-1}(r)drds\\
&=\int_0^t\int_0^{t-r}e^{-\gamma s}\psi_{k-1}(r)dsdr\\
&=\int_0^t\int_0^{t-s}e^{-\gamma s}\psi_{k-1}(r)drds\\
&=\int_0^te^{-\gamma s}\psi_{k}(t-s)ds=\varphi_{k+1}(t).
\end{align}
This is the first claim of the lemma.

The second claim of the lemma is true for $j=0$ and any $k\in\mathbb{N}$ by definition. By induction we get
\begin{align}
\int_0^t\psi_k(s)\psi_j(t-s)ds&=\int_0^t\psi_k(s)\int_0^{t-s}\psi_{j-1}(r)drds\\
&=\int_0^t\int_0^{t-r}\psi_k(s)\psi_{j-1}(r)dsdr\\
&=\int_0^t\psi_{k+1}(t-r)\psi_{j-1}(r)dr\\
&=\int_0^t\psi_{k+j}(r)\psi_0(t-r)dr=\varphi_{k+j+1}(t).
\end{align}
This completes the proof of the lemma.
\end{proof}
If the function $f$ is twice continuously differentiable, then, for small
values of $s$, the value $\nabla f(\bL_s)$ appearing in \eqref{Vt} can be approximated by an affine function of $\bL_s$:
\begin{align}
\nabla f(\bL_s)
	&\approx \nabla f(\bL_0) + \nabla^2 f(\bL_0) (\bL_s-\bL_0)\\
	& = \nabla f(\bL_0) + \nabla^2 f(\bL_0) \int_0^s \bV_w\,dw\\
	& \approx \nabla f(\bL_0) + \psi_1(s)\nabla^2 f(\bL_0) \bV_0 + \sqrt{2\gamma}\,
	\nabla^2f(\bL_0) \int_0^s \psi_1(s-w)\,d\bW_w.\label{approx_grad}
\end{align}
From the above
approximation, we can infer that
\begin{align}\label{approx}
\int_0^t e^{-\gamma (t-s)} \nabla f(\bL_s)\,ds
	&\approx  \psi_1(t)\nabla f(\bL_0) + \varphi_2(t)\nabla^2 f(\bL_0) \bV_0\\
	&\qquad +
	 \sqrt{2\gamma}\, \nabla^2f(\bL_0) \int_0^te^{-\gamma (t-s)}\int_0^s \psi_1(s-w)\,
	d\bW_w\,ds\\
	&= \psi_1(t)\nabla f(\bL_0) + \varphi_2(t)\nabla^2 f(\bL_0) \bV_0+
	\sqrt{2\gamma}\,\nabla^2f(\bL_0) \int_0^t \varphi_2(t-w)\,d\bW_w.
\end{align}
In the last step of the above equation, we have used that
\begin{align}
\int_0^te^{-\gamma (t-s)}\int_0^s \psi_1(s-w)\,d\bW_w\,ds
	& = \int_0^t \int_{w}^t e^{-\gamma (t-s)} \psi_1(s-w)\,ds\,d\bW_w\\
	& = \int_0^t \int_{0}^{t-w} e^{-\gamma (t-w-u)} \psi_1(u)\,du\,d\bW_w\\
	& = \int_0^t \varphi_2(t-w)\,d\bW_w.
\end{align}
Combining the last approximation and the diffusion equation \eqref{Vt},
we arrive at
\begin{align}
   {\bV}_t
				& \approx e^{-\gamma t}{\bV}_0 -  \psi_1(t)\nabla f(\bL_0)
					- \varphi_2(t)\nabla^2 f(\bL_0) \bV_0\\
					&\qquad -\sqrt{2\gamma}\,\nabla^2 f(\bL_0)\int_0^t \varphi_2(t-s)\,d\bW_s+
					\sqrt{2\gamma}\int_0^t e^{-\gamma (t-s)}\,d\bW_s.
\end{align}
This approximation will be used for defining the discretized version of the process $\bV$.
In order to define the discretized version of $\bL$, we will simply use the plug-in approximation of $\bV$, and then integrate.
This leads to
\begin{align}
	\bL_t &= \bL_0 + \int_0^t\bV_s\,ds\\
	&\approx \bL_0 + \psi_1(t)\bV_0 - \psi_2(t)\nabla f(\bL_0)
		- \varphi_3(t)\nabla^2f(\bL_0)\bV_0\\
	&\qquad - \sqrt{2\gamma}\,\nabla^2f(\bL_0)\int_0^t
	\varphi_3(t-w)\,d\bW_w +\sqrt{2\gamma}\int_0^t\psi_1(t-w)\,d\bW_w.
\end{align}
\subsection{Proof of \Cref{th:4}}

Recall that we have defined in \Cref{sec:2ndorderKLMC} the following functions
$$
\varphi_{k+1}(t)=\int_0^te^{-\gamma(t-s)}\psi_k(s)ds,\qquad k\ge 1.
$$
We first evaluate the error of one iteration  of the KLMC2 algorithm.
To this end, we introduce the processes
\begin{multline}
    \Tilde{\bV}_t=e^{-\gamma t}\Tilde{\bV}_0-\left(\psi_1(t)\nabla f(\Tilde{\bL}_0)
		 + \varphi_2(t)\nabla^2f(\Tilde{\bL}_0)\Tilde{\bV}_0\right)\\
		 +\sqrt{2\gamma }\left(\int_0^te^{-\gamma(t-s)}d\bW_s-\nabla^2f(\Tilde{\bL}_0)\int_0^t\varphi_2(t-s)d\bW_s\right)
\end{multline}
and
\begin{multline}
    \Tilde{\bL}_t= \Tilde{\bL}_0+\psi_1(t)\Tilde{\bV}_0-\left(\psi_2(t)\nabla f(\Tilde{\bL}_0)+
		\varphi_3(t)\nabla^2f(\Tilde{\bL}_0)\Tilde{\bV}_0\right)\\
		+\sqrt{2\gamma }\left(\int_0^t\psi_1(t-s)d\bW_s-\nabla^2f(\Tilde{\bL}_0)
		\int_0^t\varphi_3(t-s)d\bW_s\right).
\end{multline}
In what follows, we will use the following matrices to perform a linear
transformation of the space $\RR^{2p}$:
\begin{align}\label{pp-1}
\bfP=\gamma^{-1}\cdot\begin{bmatrix}
      \mathbf 0_{p\times p} & -\gamma \bfI_{p} \\
      \bfI_{p} & \bfI_{p}
\end{bmatrix},\qquad
\bfP^{-1}=\begin{bmatrix}
     \bfI_{p}  & \gamma \bfI_{p} \\
     - \bfI_{p} &  \mathbf 0_{p\times p}
\end{bmatrix}.
\end{align}
We need an auxiliary process, denoted by $(\hat\bV,\hat\bL)$, which
at time 0 coincides with $(\bV,\bL)$ but evolves according to exactly
the same dynamics as $(\tilde\bV,\tilde\bL)$.

\begin{proposition}\label{prop:2}
  Assume that, for some constants $m,M,M_2>0$, the function $f$ is
$m$-strongly convex, its gradient is $M$-Lipschitz, and its Hessian is $M_2$-Lipschitz for the spectral norm. If the parameter $\gamma$ and the step size $t$ of the kinetic Langevin diffusion are such that
  $$t\le\frac{1}{5\gamma},$$
  then
$$
 \left\|\bfP^{-1}\begin{bmatrix}
     \bV_t-\hat{\bV}_t \\
      \bL_t-\hat{\bL}_t
\end{bmatrix}\right\|_{{\LL}_2}\le0.25\times t^3(M_2\sqrt{p^2+2p}+M^{3/2}\sqrt{p}).
$$
\end{proposition}

\begin{proof}
From the definition of $\bfP^{-1}$, we compute
\begin{align}
    \bigg\|\bfP^{-1}\begin{bmatrix}
     \bV_t-\hat{\bV}_t \\
      \bL_t-\hat{\bL}_t
\end{bmatrix}\bigg\|_{{\LL}_2}=&\left\{\|\bV_t-\hat{\bV}_t+\gamma(\bL_t-\hat{\bL}_t)\|_{{\LL}_2}^2+\|\bV_t-\hat{\bV}_t\|_{{\LL}_2}^2\right\}^{1/2}\\
\le&\left\{\left(\|\bV_t-\hat{\bV}_t\|_{{\LL}_2}+\gamma\|\bL_t-\hat{\bL}_t\|_{{\LL}_2}\right)^2+\|\bV_t-\hat{\bV}_t\|_{{\LL}_2}^2\right\}^{1/2}
\end{align}
where the upper bound follows from Minkowski's inequality. We now give upper bounds for the ${\LL}_2$-norm of processes $\bV-\hat{\bV}$ and $\bL-\hat{\bL}$.
\begin{lemma}
For any time step $t>0$ we have
$$\|\hat{\bV}_t-\bV_t\|_{{\LL}_2}\le\frac{t^3(M_2\sqrt{p^2+2p}+M^{3/2}\sqrt{p})}{6},$$
$$\|\hat{\bL}_t-\bL_t\|_{{\LL}_2}\le\frac{t^4(M_2\sqrt{p^2+2p}+M^{3/2}\sqrt{p})}{24}.
$$
\end{lemma}

\begin{proof}
Recall that $\psi_1(t)=\int_0^t e^{-\gamma(t-s)}ds$, $\psi_2(t)=\int_0^tse^{-\gamma(t-s)}ds$ and
$$
\bV_t=e^{-\gamma t} \bV_0-\int_0^t e^{-\gamma (t-s)}\nabla f(\bL_s)ds+\sqrt{2\gamma}\int_0^te^{-\gamma (t-s)}d\bW_s.
$$
We compute
\begin{multline}
\hat{\bV}_t-\bV_t=
\int_0^te^{-\gamma(t-s)}(\nabla f(\bL_s)-\nabla f(\bL_0))ds
-\varphi_2(t)\nabla^2f(\bL_0) \bV_0\\
-\sqrt{2\gamma}\nabla^2 f(\bL_0)\int_0^t\varphi_2(t-s)d\bW_s.
\end{multline}
 By Taylor's theorem, we have
 $$
 \nabla f(\bL_s)-\nabla f(\bL_0)=\bfH_s\cdot(\bL_s-\bL_0),\qquad\bfH_s\triangleq\int_0^1\nabla^2f(\bL_s+h(\bL_0-\bL_s))dh.
 $$
This yields the following convenient re-writing of the first integral
\begin{multline}
    \int_0^te^{-\gamma(t-s)}(\nabla f(\bL_s)-\nabla f(\bL_0))ds\\
    =\underbrace{\int_0^te^{-\gamma(t-s)}(\bfH_s-\nabla^2 f(\bL_0))(\bL_s-\bL_0)ds}_{\triangleq \bA_t}
    +\underbrace{\nabla^2 f(\bL_0)\int_0^t\int_0^se^{-\gamma(t-s)}\bV_r drds}_{\triangleq \bC_t}.
\end{multline}
Now, we replace $\bV_r$ by its explicit expression
$$
\bV_r
		= e^{-\gamma r}\bV_0 -  \int_0^r e^{-\gamma (r-w)} \nabla f(\bL_w)\,dw+
		\sqrt{2\gamma}\int_0^r e^{-\gamma (r-w)}\,d\bW_w.
$$
By integrating twice, we compute
\begin{multline}
    \bC_t= \varphi_2(t)\nabla^2f(\bL_0) \bV_0+\sqrt{2\gamma}\nabla^2 f(\bL_0)\int_0^t\varphi_2(t-s)d\bW_s\\-\underbrace{\nabla^2 f(\bL_0)\int_0^t\int_0^s\int_0^re^{-\gamma (t-s)}e^{-\gamma (r-w)}\nabla f(\bL_w)dwdrds}_{\triangleq \bB_t}
\end{multline}
Summing the two expressions allows some terms to cancel out leading to
$$
\hat{\bV}_t-\bV_t=\bA_t-\bB_t,
$$
where
\begin{align}
    \bA_t&=\int_0^t\int_0^1e^{-\gamma (t-s)}\left(\nabla^2f(\bL_s+h(\bL_0-\bL_s))-\nabla^2f(\bL_0)\right)\cdot(\bL_s-\bL_0)dhds,\\
   \bB_t&=\nabla^2 f(\bL_0)\int_0^t\int_0^s\int_0^re^{-\gamma (t-s)}e^{-\gamma (r-w)}\nabla f(\bL_w)dwdrds.
\end{align}
We now control ${\LL}_2$-norm of processes $\bA_t$ and $\bB_t$. Bounding $e^{-\gamma (t-s)}$ by one, Minkowski's inequality in its integral version and the Lipschitz assumption on the Hessian yield
\begin{align}
    \|\bA_t\|_{{\LL}_2}&\le\int_0^t\int_0^1\bfE\Big[\|\left(\nabla^2f(\bL_s+h(\bL_0-\bL_s))-\nabla^2f(\bL_0)\right)\cdot(\bL_s-\bL_0)\|_2^2\Big]^{1/2}dhds\\
    &\le M_2\int_0^t\int_0^1\bfE\Big[(1-h)^2\|\bL_s-\bL_0\|_2^4\Big]^{1/2}dhds\\
    &=\frac{M_2}{2}\int_0^t\bigg\{\bfE\left[\left\|\int_0^s\bV_r dr\right\|_{2}^4\right]^{1/4}\bigg\}^2ds\\
		&\le\frac{M_2}{2}\int_0^t\bigg\{\int_0^s\bfE\left[\left\|\bV_r \right\|_{2}^4\right]^{1/4}dr\bigg\}^2ds\\
		&=\frac{M_2}{2}\int_0^t\bigg\{\int_0^s\bfE\left[\left\|\bV_0 \right\|_{2}^4\right]^{1/4}dr\bigg\}^2ds\\
    &=\frac{M_2t^3}{6}\bfE\left[\|\bV_0\|_2^4\right]^{1/2},
\end{align}
where we have used the stationarity of the process $\bV_r$.
Since $\bV_0$ is standard Gaussian, we get $\bfE\left[\|\bV_0\|_2^4\right]=
p^2+2p$.

In the same way, Minkowski's inequality in its integral version yields
\begin{align}
    \|\bB_t\|_{{\LL}_2}&\le\int_0^t\int_0^s\int_0^r\|\nabla^2 f(\bL_0)\nabla f(\bL_w)\|_{{\LL}_2}dwdrds\\
    &\le \int_0^t\int_0^s\int_0^rM\|\nabla f(\bL_w)\|_{{\LL}_2}dwdrds\\
    &=M\|\nabla f(\bL_0)\|_{{\LL}_2}\int_0^t\int_0^s\int_0^rdwdrds\\
    &=\frac{t^3M}{6}\|\nabla f(\bL_0)\|_{{\LL}_2},
\end{align}
where last equalities follow from the stationarity of $\bL_w$.
Since $\bL_0\sim\pi$ \citep[Lemma 2]{DalalyanColt} ensures that $\|\nabla f(\bL_0)\|_{{\LL}_2}\le\sqrt{Mp}$, and the first claim of the lemma follows.

The bound for process $\bL-\hat{\bL}$ follows from Minkowski's inequality combined with the bound just proven:
\begin{align}
    \|\hat{\bL}_t-\bL_t\|_{{\LL}_2}&\le\int_0^t\|\hat{\bV}_s -\bV_s\|_{{\LL}_2}ds\\
    &\le\int_0^t\left(\frac{t^3(M_2\sqrt{p^2+2p}+M^{3/2}\sqrt{p})}{6}\right)ds\\
    &=\frac{t^4(M_2\sqrt{p^2+2p}+M^{3/2}\sqrt{p})}{24}.
\end{align}
This completes the proof of the lemma.
\end{proof}
The claim of the proposition follows from the assumption $\gamma t\le1/5$
and that
$$
\sqrt{\left(\frac16+\frac{1}{5\times24}\right)^2+\left(\frac16\right)^2}\le0.25
$$
\end{proof}

The next, perhaps the most important, step of the proof is to assess the distance
between the random vectors $(\hat\bV_t,\hat\bL_t)$ and $(\tilde\bV_t,\tilde\bL_t)$.
\begin{proposition}\label{prop:3}
 Assume that, for some constants $m,M,M_2>0$, the function $f$ is
$m$-strongly convex, its gradient is $M$-Lipschitz, and its Hessian is $M_2$-Lipschitz for the spectral norm. If the parameter $\gamma$ and the step size $t$ of the kinetic Langevin diffusion satisfy the inequalities
$$
\gamma^2\ge m+M,\qquad t\le\frac{1}{5\gamma\varkappa},
$$
then, for  the $(2p)\times (2p)$ matrix $\bfP$ defined in \eqref{pp-1}, and for every
$a\ge 5p$, it holds
  \begin{align}
       \left\|\bfP^{-1}\begin{bmatrix}
     \hat{\bV}_t - \Tilde\bV_t \\
      \hat{\bL}_t - \Tilde\bL_t
\end{bmatrix}\right\|_{{\LL}_2}
&\le
\left(1-\frac{mt}{2\gamma}+\frac{M_2\sqrt{a}\,t^2}{\gamma}\right)
\left\|\bfP^{-1}\begin{bmatrix}
     \bV_0-\Tilde{\bV}_0  \\
      \bL_0-\Tilde{\bL}_0
\end{bmatrix}\right\|_{{\LL}_2}\\
&\qquad + \sqrt{2}\,t^2(M-m)e^{-(a-p)/8}.
  \end{align}
\end{proposition}

\begin{proof}
\textbf{Step 1:}
After change of basis, the new discretized process rewrites:
\begin{align}
    \bfP^{-1}\begin{bmatrix}
     \hat{\bV}_t - \Tilde\bV_t \\
      \hat{\bL}_t - \Tilde\bL_t
\end{bmatrix}
&=
\Big\{\bfI_{2p}-\psi_1(t)\underbrace{\bfP^{-1}\bfR_0\bfP}_{\triangleq \bfQ_0}-\underbrace{\bfP^{-1}\bfE_0(t)\bfP}_{\triangleq \bfN_0(t)}\Big\}
\cdot\bfP^{-1}\begin{bmatrix}
     \bV_0-\Tilde{\bV}_0  \\
      \bL_0-\Tilde{\bL}_0
\end{bmatrix}\\
&+\bfP^{-1}\begin{bmatrix}
    \varphi_2(t)(\nabla^2 f(\bL_0)-\nabla^2 f(\Tilde{\bL}_0))\bV_0 \\
    \varphi_3(t)(\nabla^2 f(\bL_0)-\nabla^2 f(\Tilde{\bL}_0))\bV_0
\end{bmatrix},
\end{align}
where
$$
\bfR_0=\begin{bmatrix}
     \gamma\bfI_p & \bfH_0 \\
      -\bfI_p & \mathbf 0_{p\times p}
\end{bmatrix},\qquad
\bfE_0(t)\triangleq \begin{bmatrix}
     \varphi_2(t)\nabla^2f(\Tilde{\bL}_0) & \mathbf 0_{p\times p}\\
     \varphi_3(t)\nabla^2f(\Tilde{\bL}_0) & -\psi_2(t)\bfH_0
\end{bmatrix}.
$$
By Minkowski's inequality and the definition of $\bfP^{-1}$, we get
\begin{align}
    \left\|\bfP^{-1}\begin{bmatrix}
     \hat{\bV}_t - \Tilde\bV_t \\
      \hat{\bL}_t - \Tilde\bL_t
\end{bmatrix}\right\|_{{\LL}_2}&\le\left\|\Big\{\bfI_{2p}-\psi_1(t) \bfQ_0- \bfN_0(t)\Big\}
\cdot\bfP^{-1}\begin{bmatrix}
     \bV_0-\Tilde{\bV}_0  \\
      \bL_0-\Tilde{\bL}_0
\end{bmatrix}\right\|_{{\LL}_2}\\
&+\xi_2(t)\left\|\left(\nabla^2 f(\bL_0)-\nabla^2 f(\Tilde{\bL}_0)\right)\cdot\bV_0\right\|_{{\LL}_2}
\end{align}
where
$$
 \xi_2(t)\triangleq\sqrt{\big(\varphi_2(t)+\gamma\varphi_3(t)\big)^2+\varphi_2(t)^2}.
$$
We have
\begin{align}
\varphi_2(t)+\gamma\varphi_3(t)
&= \int_0^t e^{-\gamma(t-s)} (\psi_1(s)+\gamma\psi_2(s))\,ds\\
&= \frac1\gamma\int_0^t e^{-\gamma(t-s)} (1-e^{-\gamma s}+s\gamma - 1+e^{-\gamma s})\,ds
\le t^2/2.
\end{align}
Therefore,
$\xi_2(t)\le t^2/\sqrt{2}$.

\textbf{Step 2:} We give an upper bound for the following spectral norm
$$
\|\bfI_{2p}-\psi_1(t) \bfQ_0- \bfN_0(t)\|\le \|\bfI_{2p}-\psi_1(t) \bfQ_0\|+\|\bfN_0(t)\|.
$$
One can check that
$$\|\bfI_{2p}-\psi_1(t)\bfQ_0\|\le1-\psi_1(t)(m/\gamma)+0.5\psi_1(t)^2M(\alpha+m^2/(M\gamma^2))$$
where $\alpha=\max(1-M/\gamma^2,3M/\gamma^2-1)$.

Direct calculation yields
\begin{align}
\bfN_0(t)&=\bfP^{-1}\bfE_0(t)\bfP\\
&=
\gamma^{-1}\begin{bmatrix}
     \bfI_{p}  & \gamma \bfI_{p} \\
     - \bfI_{p} &  \mathbf 0_{p\times p}
\end{bmatrix}
\begin{bmatrix}
     \varphi_2(t)\nabla^2f(\Tilde{\bL}_0) & \mathbf 0_{p\times p}\\
     \varphi_3(t)\nabla^2f(\Tilde{\bL}_0) & -\psi_2(t)\bfH_0
\end{bmatrix}
\begin{bmatrix}
      \mathbf 0_{p\times p} & -\gamma \bfI_{p} \\
      \bfI_{p} & \bfI_{p}
\end{bmatrix}\\
&=
\gamma^{-1}
\begin{bmatrix}
     (\varphi_2+\gamma\varphi_3)\nabla^2f(\Tilde{\bL}_0) & \mathbf 0\\
     -\varphi_2\nabla^2f(\Tilde{\bL}_0) & \mathbf 0
\end{bmatrix}
\begin{bmatrix}
      \mathbf 0 & -\gamma \bfI_{p} \\
      \bfI_{p} & \bfI_{p}
\end{bmatrix}
-\psi_2(t)\gamma^{-1}
\begin{bmatrix}
     \mathbf 0 & \gamma \bfH_0\\
     \mathbf 0 & \mathbf 0
\end{bmatrix}
\begin{bmatrix}
      \mathbf 0 & -\gamma \bfI_{p} \\
      \bfI_{p} & \bfI_{p}
\end{bmatrix}\\
&=\begin{bmatrix}
      \mathbf 0_{p\times p} & -(\varphi_2(t)+\gamma\varphi_3(t))\nabla^2f
			(\Tilde{\bL}_0) \\
      \mathbf 0_{p\times p} & \varphi_2(t)\nabla^2f(\Tilde{\bL}_0)
\end{bmatrix}-\psi_2(t)\begin{bmatrix}
     \bfH_0 & \bfH_0 \\
      \mathbf 0_{p\times p} &  \mathbf 0_{p\times p}
\end{bmatrix}.
\end{align}
Since $\nabla^2f(\Tilde{\bL}_0)$ and $\bfH_0$ are both upper bounded by $M\bfI_p$
and $0\le \psi_2(t)\le\varphi_2(t)\le\varphi_2(t)+\gamma\varphi_3(t)\le t^2/2$, we get
$$
\left\|\bfN_0(t)\right\|\le \sqrt{2}M t^2.
$$
Summing the two upper bounds, we get
$$
\left\|\bfI_{2p}-\psi_1(t)\bfQ_0-\bfN_0(t)\right\|\le\rho_t
\triangleq\left\{1-\frac{\psi_1(t)m}{\gamma}+\frac{\psi_1(t)^2M}{2}
\left(\alpha+\frac{m^2}{M\gamma^2}\right)+
M\sqrt2\, t^2\right\}.$$
Taylor's expansion ensures that $t-\gamma t^2/2\le\psi_1(t)\le t$ and, therefore,
$$
\rho_t\le1-\frac{mt}{\gamma}+\frac{Mt^2}{2}\underbrace{\left(\alpha+\frac{m^2}{M\gamma^2}+\frac{m}{M}+2\sqrt{2}\right)}_{\le 2+2\sqrt{2}\le5}.
$$
Finally, we use the condition $t\le 1/(5\gamma\varkappa)$
to bound $\rho_t$ by $1-mt/(2\gamma)$.

\textbf{Step 3:} We control the ${\LL}_2$-norm of $(\nabla^2 f(\bL_0)-\nabla^2 f(\Tilde{\bL}_0))\bV_0$.

Since $m\bfI_p\preccurlyeq\nabla^2 f(x)\preccurlyeq M\bfI_p$, combined with the fact that the Hessian is $M_2$-Lipschitz, we get
$$
\left\|\left(\nabla^2 f(\bL_0)-\nabla^2 f(\Tilde{\bL}_0)\right)\cdot\bV_0\right\|_2\le \min\left(M-m,M_2\|\bL_0-\Tilde{\bL}_0\|_2\right)\|\bV_0\|_2.
$$
Using the obvious inequality $\|\bV_0\|_2^2\le a + (\|\bV_0\|_2^2-a)_+$, for every $a>0$, 
this implies that
\begin{align}
    \bfE\left[\left\|\left(\nabla^2 f(\bL_0)-\nabla^2 f(\Tilde{\bL}_0)\right)
		\,\bV_0\right\|_2^2\right]
				&\le \bfE\left[\min\left((M-m)^2,M_2^2\|\bL_0-\Tilde{\bL}_0\|^2_2\right)
						\|\bV_0\|_2^2\right]\\
				&\le M_2^2a\,\bfE\left[\|\bL_0-\Tilde{\bL}_0\|^2_2\right]
						+(M-m)^2\bfE\left[(\|\bV_0\|_2^2-a)_+\right]\\
				&\stackrel{(1)}{\le} M_2^2a\,\|\bL_0-\Tilde{\bL}_0\|^2_{\LL_2}
						+ 4(M-m)^2e^{-(a-p)/4}, 
\end{align}
where inequality (1) is valid for every $a\ge 5p$ according to well-known bounds
on the $\chi^2$ distribution; see for instance \citep[Lemmas 5-6]{ColDal17a}. 
Finally, recall that
$$
\|\bL_0-\Tilde{\bL}_0\|_{{\LL}_2}\le\gamma^{-1}\sqrt{2}\left\|\bfP^{-1}\begin{bmatrix}
     \bV_0-\Tilde{\bV}_0  \\
      \bL_0-\Tilde{\bL}_0
\end{bmatrix}\right\|_{{\LL}_2}.
$$
Taking square roots yields the claim of the proposition.
\end{proof}
%%%
The last piece of the proof is the following proposition.
%%%
\begin{proposition}\label{prop:4}
  Assume that, for some constants $m,M,M_2>0$, the function $f$ is
$m$-strongly convex, its gradient is $M$-Lipschitz, and its Hessian is $M_2$-Lipschitz for the spectral norm. If the parameter $\gamma$ and the step size $h$ of the kinetic Langevin diffusion satisfy the inequalities
$$
\gamma^2\ge m+M,\qquad h\le\frac{1}{5\gamma\varkappa}\wedge \frac{m}{4\sqrt{5p}\,M_2}.
$$
  Then
  \begin{align}
      \left\|\bfP^{-1}\begin{bmatrix}
     {\bV}_{kh} - \Tilde\bV_{kh} \\
      {\bL}_{kh} - \Tilde\bL_{kh}
\end{bmatrix}\right\|_{{\LL}_2}
\le&
\left(1-\frac{mh}{4\gamma}\right)^k
\left\|\bfP^{-1}\begin{bmatrix}
     \bV_0-\Tilde{\bV}_0  \\
      \bL_0-\Tilde{\bL}_0
\end{bmatrix}\right\|_{{\LL}_2}
\\
&+\frac{4\sqrt{2}\,(M-m)}{m}\,\gamma h e^{-\frac{m^2}{160 M_2^2h^2}}
+\gamma h^2\left(\frac{M_2}{m}\sqrt{p^2+2p}+\frac{M^{3/2}}{m}\sqrt{p}\right).
  \end{align}
\end{proposition}

\begin{proof}
Minkowski's inequality yields
$$
\left\|\bfP^{-1}\begin{bmatrix}
     {\bV}_{kh} - \Tilde\bV_{kh} \\
      {\bL}_{kh} - \Tilde\bL_{kh}
\end{bmatrix}\right\|_{{\LL}_2}\le\left\|\bfP^{-1}\begin{bmatrix}
     \hat{\bV}_{kh} - \Tilde\bV_{kh} \\
      \hat{\bL}_{kh} - \Tilde\bL_{kh}
\end{bmatrix}\right\|_{{\LL}_2}+\left\|\bfP^{-1}\begin{bmatrix}
     {\bV}_{kh} - \hat\bV_{kh} \\
      {\bL}_{kh} - \hat\bL_{kh}
\end{bmatrix}\right\|_{{\LL}_2}.
$$
For $k\ge 0$, define
$$
x_k=  \left\|\bfP^{-1}\begin{bmatrix}
     {\bV}_{kh} - \Tilde\bV_{kh} \\
      {\bL}_{kh} - \Tilde\bL_{kh}
\end{bmatrix}\right\|_{{\LL}_2}.
$$
By \Cref{prop:2} and \Cref{prop:3}, we thus have
$$
x_{k+1}\le \left(1-\frac{mh}{2\gamma}+\frac{M_2\sqrt{a}\,h^2}{\gamma}\right)
x_k + \sqrt{2}\,h^2(M-m)e^{-(a-p)/8}
+0.25h^3\left(M_2\sqrt{p^2+2p}+M^{3/2}\sqrt{p}\right).
$$
Assuming that $\sqrt{a} = m/(4M_2h)\ge \sqrt{5p}$ and unfolding the last recursion, we get 
\begin{align}
x_{k+1}\le \left(1-\frac{mh}{4\gamma}\right)^{k+1}
x_0 + \frac{4\sqrt{2}\,(M-m)}{m}\,\gamma h e^{-(a-p)/8}
+\gamma h^2\left(\frac{M_2}{m}\sqrt{p^2+2p}+\frac{M^{3/2}}{m}\sqrt{p}\right).
\end{align}
Easy algebra shows that
$$
\frac{a-p}{8} = \frac{a}{10} + \frac{a-5p}{40} \ge \frac{a}{10} = \frac{m^2}{160 M_2^2h^2}.
$$
This is exactly the claim of the proposition.
\end{proof}

To complete the proof of \Cref{th:4}, we need to do some
simple algebra. First of all, using the relations
\begin{align}
W_2(\nu_k,\pi) \le
       {\gamma^{-1}\sqrt{2}}\left\|\bfP^{-1}\begin{bmatrix}
     {\bV}_{kh} - \Tilde\bV_{kh} \\
      {\bL}_{kh} - \Tilde\bL_{kh}
\end{bmatrix}\right\|_{{\LL}_2},\qquad
W_2(\nu_0,\pi) = \gamma^{-1}\,
       \left\|\bfP^{-1}\begin{bmatrix}
     {\bV}_{0} - \Tilde\bV_{0} \\
      {\bL}_{0} - \Tilde\bL_{0}
\end{bmatrix}\right\|_{{\LL}_2}
\end{align}
as well as the inequality $p^2+2p\le 2p^2$ (since $p\ge 2$), we arrive at
\begin{align}
		W_2(\nu_k,\pi)
		&\le \sqrt2\left(1-\frac{mh}{4\gamma}\right)^k W_2(\nu_0,\pi)\\
		&\qquad +\frac{8\,(M-m)}{m}\, h e^{-\frac{m^2}{160 M_2^2h^2}}
+\sqrt{2}\,h^2\left(\frac{M_2p}{m}\sqrt{2}+\frac{M^{3/2}}{m}\sqrt{p}\right).\label{s10:1}
\end{align}
This leads to the claim of the theorem.

\section*{Acknowledgments}
The work of AD was partially supported by the grant
Investissements d'Avenir (ANR-11-IDEX-0003/Labex Ecodec/ANR-11-LABX-0047).

{\renewcommand{\addtocontents}[2]{}
\bibliography{Literature1}}

\end{document}